\documentclass[12pt,a4paper,reqno]{amsart}    

\usepackage{geometry} 
\usepackage{float} 
\usepackage[headings]{fullpage} 

\usepackage[french,english]{babel}

\usepackage{graphicx}
\usepackage{caption} 
\usepackage{url}  

\newtheorem{thm}{Theorem}[section]
\newtheorem{Con}[thm]{Conjecture}

\newtheorem{cor}[thm]{Corollary}
\newtheorem{lem}[thm]{Lemma}
\newtheorem{pro}[thm]{Proposition}
\theoremstyle{definition}

\newtheorem{rem}[thm]{Remark}

\numberwithin{equation}{section}
 
\newcommand{\C}{\mathbb{C}} 
\newcommand{\D}{\mathbb{D}} 
\newcommand{\Q}{\mathbb{Q}} 
\newcommand{\X}{\mathbb{X}}
\newcommand{\Y}{\mathbb{Y}}
\newcommand{\ex}{\mathbb{E}}
\newcommand{\E}{\mathcal{E}}
\newcommand{\re}{\textup{Re}}
\newcommand{\im}{\textup{Im}}
\newcommand{\pr}{\mathbb{P}}

\newcommand{\GG}{\mathfrak{G}}
\newcommand{\F}{\mathcal{F}}

\newcommand{\ld}{\textup{Ld}}

\newcommand{\Lr}{\textup{Ld}(1, \X)}

\usepackage{bm} 
\usepackage{mathrsfs}

\usepackage[full]{textcomp} 
\usepackage{newtxtext} 
\usepackage{cabin} 
\usepackage{zlmtt}
\usepackage[bigdelims,vvarbb]{newtxmath} 
\usepackage[cal=boondoxo]{mathalfa} 

\usepackage{microtype} %

\allowdisplaybreaks

\newcommand{\bound}{10^7}   
\begin{document}

\title{Small values of $\vert L^\prime/L(1,\chi) \vert$}

\author{Youness Lamzouri}
\author{ Alessandro Languasco}
%


\begin{abstract}    
In this paper, we investigate the quantity $m_q:=\min_{\chi\ne \chi_0} \vert L^\prime/L(1,\chi)\vert$, as $q\to \infty$ over the primes, where $L(s,\chi)$ is the Dirichlet $L$-function attached to a non trivial Dirichlet character modulo $q$. Our main result shows that $m_q \ll \log\log q/\sqrt{\log q}$. We also compute $m_q$
 for every odd prime $q$ up to $\bound$. As a consequence we numerically verified that for every odd prime $q$,
 $3 \le q \le \bound$, we have
$c_1/q< m_q<5/\sqrt{q}$, with $c_1=21/200$.
 In particular, this shows that $L^\prime(1,\chi) \ne 0$  for every
non trivial Dirichlet character $\chi$ mod $q$ where $3\leq q\leq \bound$ is prime, answering a question of Gun, Murty and Rath in this range.
We  also provide some statistics and scatter plots regarding the $m_q$-values, see Section \ref{tables-figures}.
The programs used and the computational results described here
are available at the following web address:
\url{http://www.math.unipd.it/~languasc/smallvalues.html}.
\end{abstract} 
\maketitle 

\section{Introduction}
Let $q$ be an odd prime, $\chi$ be a
non-trivial Dirichlet character mod $q$, and  $L(s,\chi)$ be the Dirichlet $L$-function attached to $\chi$. 
We also denote by $\chi_0$ the trivial Dirichlet character mod $q$.
It is well known that the size  of the logarithmic derivative of $L(s,\chi)$ at $1$ is 
connected with the distribution of its non-trivial zeros; moreover, its average over non trivial characters
was recently studied by Ihara in his papers \cite{Ihara2006,Ihara2008} about the  \emph{Euler-Kronecker}
constant  for number fields. In particular, denoting
by $\zeta_q$  a primitive $q$-th root of unity and 
$\zeta_{\Q(\zeta_q)}(s)$ the Dedekind zeta-function of $\Q(\zeta_q)$,
 the expansion of $\zeta_{\Q(\zeta_q)}(s)$ near $s=1$ is
\[
\zeta_{\Q(\zeta_q)} (s) = \frac{c_{-1}}{s-1} + c_0 + O(s-1),
\]
and the Euler-Kronecker constant of $\Q(\zeta_q)$ is defined as  
\[
\lim_{s\to 1} \Bigl(\frac{\zeta_{\Q(\zeta_q)} (s)}{c_1} - \frac{1}{s-1}\Bigr)=
\frac{c_0}{c_{-1}}.
\]
Recalling that 
$\zeta_{\Q(\zeta_q)} (s)= \zeta(s) \prod_{\chi \neq \chi_0} L(s,\chi)$,
where $\zeta(s)$ is the Riemann zeta-function,
by logarithmic differentiation we immediately get that
the \emph{Euler-Kronecker constant
for the prime cyclotomic field $\Q(\zeta_q)$} is 
\begin{equation}
\label{EKq-def}
\GG_{q} 
: =
\gamma
+
\sum_{\chi \neq \chi_0}\frac{L^\prime}{L} (1,\chi).
\end{equation} 
The quantity $\GG_q$ is sometimes denoted by $\gamma_q$
but this conflicts with notations used in the literature.
Computational results on $\GG_{q}$ are developed in the papers of
 Ford-Luca-Moree \cite{FordLM2014} and Languasco 
 \cite{Languasco2019}.

These results motivate the study of extreme values of 
 $|L^\prime/L(1,\chi)|$ both theoretically and computationally.
Concerning the large values of $|L^\prime/L(1,\chi)|$, Ihara, Murty and Shimura \cite{IMS} proved that that under the assumption of the Generalised Riemann Hypothesis
we have \[
M_q:=\max_{\chi \neq \chi_0 \bmod q} \Bigl\vert \frac{L^\prime}{L} (1,\chi)\Bigr\vert\
 \leq (2+o(1))\log\log q. \]
On the other hand, by adapting the techniques of Lamzouri \cite{La}, one can show that if $q$ is a large prime then
\[M_q\geq (1+o(1))\log\log q.\]
Moreover, computational results on $M_q$ can be found
in Languasco \cite{Languasco2019} and in Languasco-Righi \cite{LanguascoR2020}.

In this paper, we investigate the small values of  $ |L^\prime/L(1,\chi)|$.
Define
\begin{equation}
\label{mq-def}
m_q:=\min_{\chi \neq \chi_0 \bmod q} \Bigl\vert \frac{L^\prime}{L} (1,\chi) \Bigr\vert\ .
\end{equation}
Then, we prove the following result 

\begin{thm}\label{Main}
Let $q$ be a large prime. Then, we have 
$$ m_q \ll \frac{\log\log q}{\sqrt{\log q}}.$$
In fact, there are at least $q(\log\log q)^2/\log q$ non-principal characters $\chi\bmod q$ such that 
$$\frac{L^\prime}{L} (1,\chi) \ll \frac{\log\log q}{\sqrt{\log q}}.$$ 
Moreover, the implicit constants are absolute and effective.
\end{thm}

Theorem \ref{Main} gives the first known non-trivial upper bound for $m_q$. Furthermore, using the algorithm developed in Languasco-Righi \cite{LanguascoR2020} together with the results of Languasco \cite{Languasco2019}, we were able to compute the values of $m_q$ for $q\leq \bound$ and obtain the following computational result. 
\begin{thm}\label{comput-mq}
For  every odd prime $q$, $3 \le q \le \bound$,  we have
$c_1/q< m_q<5/\sqrt{q}$, with $c_1=21/200$. 
\end{thm}

In particular, the lower bound in Theorem \ref{comput-mq} 
implies the following
\begin{cor}\label{corollary-nonzero}
For  every odd prime $q$ up to $\bound$ 
and for every non-trivial Dirichlet character $\chi \bmod{q}$, we have
$L^\prime(1, \chi)\ne 0$.
\end{cor}
Corollary \ref{corollary-nonzero} is connected with 
a conjecture of Gun, Murty and Rath  (see Conjecture 1.2 of
\cite{GunMR2011})
concerning the linear independence  over the algebraic closure of $\Q$ of   the
values $\log \Gamma (a/q)$, $1\le a \le q $, $(a,q)=1$.
In particular, letting
\[
Z_q := \Bigl\{ \alpha \colon \alpha = \frac{L^\prime}{L}(1,\chi)\
 \text{for some primitive character}\  \chi \bmod{q}  \Bigr\},
\]
Theorem  \ref{comput-mq} implies that $0\not \in Z_q$
for every odd prime $q$ up to $\bound$,
thus responding affirmatively to a question on page 6 of \cite{GunMR2011} in this range of $q$.

Theorem \ref{comput-mq} also suggests that the upper bound of Theorem \ref{Main}
is far from being optimal. In fact, the data on $m_q$ for $q\leq 10^7$ (see Figures \ref{fig1}, \ref{fig2} and \ref{fig3} at the end of the paper) show a remarkable fit between the maximal and minimal values of $m_q$, and the curves $b_1/\sqrt{q}$ and $c_1/q$ respectively, for some constant $b_1>0$. Based on this we make the following conjecture 
\begin{Con}\label{truth-mq}
For all $\varepsilon>0$ and for all odd primes $q$ we have 
$$ q^{-1-\varepsilon} \ll_{\varepsilon} m_q \ll_{\varepsilon} q^{-1/2+\varepsilon}.$$
In particular, for all odd primes $q$, $0\notin Z_q$.
\end{Con}

In order to prove Theorem \ref{Main}, our idea consists of studying the distribution of $L^\prime/L(1, \chi)$ as $\chi$ varies among non-principal characters modulo $q$. Indeed, we shall compare this distribution to that of an adequate probabilistic random model, which we construct as follows. Let $\{\X(p)\}_p$ be a sequence of independent random variables, indexed by the primes, and uniformly distributed on the unit circle. We extend the $\X(p)$ multiplicatively, by putting $\X(n)=\prod_{i=1}^k \X(p_i)^{a_i}$ if the prime factorization of $n$ is $n=\prod_{i=1}^k p_i^{a_i}$.   We now consider the random sum
\begin{equation}\label{SumRandomInd}
\textup{Ld}(1, \X):=- \sum_{n=1}^{\infty}\frac{\Lambda(n)\X(n)}{n}= \sum_{p}\frac{(\log p)\X(p)}{p-\X(p)},\end{equation}
where $\Lambda(n)$ denotes the von Mangoldt function.
Since $\ex(\X(n))=0$ for all $n>1$, and $\sum_{n\geq 2}\Lambda(n)^2/n^2<\infty$, it follows from Kolmogorov's three series theorem that $\Lr$ is almost surely convergent. Ihara, Murty and Shimura \cite{IMS} proved that as $q\to \infty$ through primes, the distribution of $L^\prime/L(1, \chi)$ as $\chi$ varies over non-principal characters modulo $q$, converges to that of $\ld(1, \X)$. More precisely, for any rectangle $\mathcal{R}\subset \mathbb{C}$ we have 
\begin{equation}\label{IhMuSh}
\lim_{q\to \infty} \frac{1}{q-1} \Bigl|\Bigl\{\chi\neq \chi_0 \bmod q : \frac{L^\prime}{L} (1,\chi)\in \mathcal{R} \Bigr\}\Bigr|= \pr\left(\Lr\in \mathcal{R}\right).
\end{equation}
In order to gain an understanding of how small $L^\prime/L(1, \chi)$ can be, we shall improve the results of Ihara, Murty and Shimura, by bounding the ``\emph{discrepancy}'' of the distribution of  
$L^\prime/L(1, \chi)$, which we define as 
$$
\D(q): = \sup_{\mathcal{R}} \left| \frac{1}{q-1} \Bigl|\Bigl\{\chi\neq \chi_0 \bmod q : \frac{L^\prime}{L} (1,\chi)\in \mathcal{R} \Bigr\}\Bigr|- \pr\left(\Lr\in \mathcal{R}\right)\right|,
$$
where the supremum is taken over all rectangles (possibly unbounded) of the complex plan with sides parallel to the coordinate axes. 
Using the approach of Lamzouri, Lester and  
Radziwi\l\l\ \cite{LLR}, we prove the following result, from which we shall deduce Theorem \ref{Main}.
\begin{thm}\label{Discrepancy}
Let $q$ be a large prime. Then we have 
$$ \D(q) \ll \frac{(\log\log q)^2}{\log q}.$$
\end{thm}
To establish \eqref{IhMuSh}, Ihara, Murty and Shimura investigated the moments of $L^\prime/L(1, \chi)$. For any positive integer $k$, we define 
\begin{equation}\label{DefLambda}
\Lambda_{k}(n)=\sum_{\substack{n_1,n_2,\dots,n_k\geq 1\\ n_1n_2\cdots n_k=n}}\Lambda(n_1)\Lambda(n_2)\cdots\Lambda(n_k).
\end{equation}
Then for all complex numbers $s$ with $\re(s)>1$
we have
$$ \Bigl(\frac{L^\prime}{L}(s,\chi)\Bigr)^k=(-1)^k\sum_{n=1}^{\infty} \frac{\Lambda_k(n)}{n^s}\chi(n).$$
Ihara, Murty and Shimura proved (see Theorem 5 of \cite{IMS})  that for all fixed integers $k, \ell \geq 1$ and for all $\varepsilon>0$ we have
\begin{equation}\label{IMSMoments}
\frac{1}{q-1}\sum_{\chi\neq \chi_0 \bmod q} \Bigl(\frac{L^\prime}{L}(1,\chi)\Bigr)^k\Bigl(\overline{\frac{L^\prime}{L}(1,\chi)}\Bigr)^{\ell}= (-1)^{k+\ell}\sum_{n=1}^{\infty} \frac{\Lambda_k(n) \Lambda_{\ell}(n)}{n^2} +O_{k, \ell, \varepsilon}(q^{\varepsilon-1}).
\end{equation}
Note that the main term of this asymptotic formula equals the corresponding moments of the probabilistic random model. Indeed, since $\ex(\X(n)\overline{\X(m)})=1$ if $m=n$ and equals $0$ otherwise, then for all $k, \ell\geq 1$ we have 
\begin{equation}\label{ExplicitRandomMoments}
 \ex\left(\ld(1, \X)^k\overline{\ld(1, \X)}^{\ell}\right)= (-1)^{k+\ell}\sum_{n=1}^{\infty} \frac{\Lambda_k(n) \Lambda_{\ell}(n)}{n^2}.
 \end{equation}
 \renewcommand*{\thefootnote}{\fnsymbol{footnote}}
 \setcounter{footnote}{1}
Moreover, the factor $q^{\varepsilon}$ in the error term of \eqref{IMSMoments} is due to the possible ``\emph{exceptional}'' character modulo $q$\footnote{By an exceptional character modulo a prime $q$, we mean the unique real character  $\chi_1$ (if it exists) such that  $L(s, \chi_1)$ has a zero $\rho$ with $\re(\rho) > 1-c/\log(q)$, where $c > 0$ is a fixed small constant independent of $q$.}. In order to prove Theorem \ref{Discrepancy}, we need to show that the asymptotic formula \eqref{IMSMoments} holds uniformly for $k, \ell \ll (\log q)/\log\log q$. To this end, we need to remove the possible contribution of the exceptional character $\chi_1$, as it will heavily affect the moments. 
Let 
\begin{equation}
\label{Fq-def}
\F_q:= \{ \chi\neq \chi_0 \bmod \ q : \chi \text{ is not exceptional}\}.
\end{equation}
Note that $q-2\leq |\F|\leq q-1$. 
We establish the following result, which improves \eqref{IMSMoments}.
\begin{thm}\label{MomentsLogDer}
Let $q$ be a large prime. For all positive integers $k, \ell \leq \log q/(50\log\log q)$ we have  
$$\frac{1}{q-1}\sum_{\chi\in \F_q} \Bigl(\frac{L^\prime}{L}(1,\chi)\Bigr)^k\Bigl(\overline{\frac{L^\prime}{L}(1,\chi)}\Bigr)^{\ell}= \ex\left(\ld(1, \X)^k\overline{\ld(1, \X)}^{\ell}\right)+ O\left(q^{-1/30}\right) .$$
\end{thm}

The plan of the paper is as follows. In Section \ref{RandomSec} we shall investigate the distribution of the random model $\ld(1, \X)$, and deduce Theorem \ref{Main} from Theorem \ref{Discrepancy}. In Section \ref{MomentsSec} we establish Theorem \ref{MomentsLogDer}, which gives asymptotic formulas for large moments of $L^\prime/L(1, \chi)$. These are then used in Section \ref{DiscrepancySec} to show that the characteristic function of $L^\prime/L(1, \chi)$ is very close to that of the probabilistic random model $\ld(1, \X)$. Theorem \ref{Discrepancy} will be deduced from this result using Beurling-Selberg polynomials. In Section \ref{CompSec}, we shall  present the numerical approach
we use to prove Theorem \ref{comput-mq}. Finally, in Section \ref{tables-figures}, located after the References, we shall
insert some tables and figures.

  \medskip 
\textbf{Acknowledgements}. 
The second author (A. Languasco) would like 
to thank Luca Righi (University of Padova) for his help in organising the computation
described in Section \ref{CompSec} on the
University of Padova Strategic Research Infrastructure 
Grant 2017: ``CAPRI: Calcolo ad Alte Prestazioni per la Ricerca e l'Innovazione'',
\url{http://capri.dei.unipd.it}.


\section{The distribution of $\ld(1, \X)$, and the deduction of Theorem \ref{Main}} \label{RandomSec}

The characteristic function of the joint distribution of $\re (\Lr)$ and $\im(\Lr)$ is defined by
\begin{equation}\label{DefPhiRand}
\Phi_{\textup{rand}}(u, v):= \ex\Big(\exp(iu \re (\Lr) + iv \im(\Lr)\big)\Big),
\end{equation}
for  $u, v\in \mathbb{R}$. By \eqref{SumRandomInd} it follows that 
$$ \Phi_{\textup{rand}}(u, v) =\prod_{p} \Phi_{\textup{rand}}(u, v; p), $$
where 
$$
\Phi_{\textup{rand}}(u, v; p):=\ex\left(\exp\left(iu \re \frac{(\log p)\X(p)}{p-\X(p)} + iv \im\frac{(\log p)\X(p)}{p-\X(p)} \right)\right).
$$
We first show that $\Phi_{\textup{rand}}(u, v)$ is rapidly decreasing as $|u|, |v|\to \infty$. 
\begin{pro}\label{CharDecTrunc} There exists a constant $c_0>0$ such that for all $u, v\in \mathbb{R}$ such that $|u|, |v|\geq 2$ we have 
$$
 \Phi_{\textup{rand}}(u, v) \ll \exp\big(-c_0 (|u|+|v|)\big).
$$
\end{pro}
\begin{proof}
First, note that for all primes $p$ and all $u, v \in \mathbb{R}$ we have 
$|\Phi_{\textup{rand}}(u, v; p)|\leq 1.$ Hence, we get
\begin{equation}\label{RangePrimes}
|\Phi_{\textup{rand}}(u, v)| \leq \prod_{p\geq X} |\Phi_{\textup{rand}}(u, v; p)|,
\end{equation}
for any parameter $ X\geq 2$. Furthermore, observe that 
$$ \frac{(\log p)\X(p)}{p-\X(p)} = \frac{\log p}{p}\X(p) + O\Bigl(\frac{\log p}{p^2}\Bigr).
$$
This implies 
$$ \Phi_{\textup{rand}}(u, v; p)= \ex\left(\exp\left(iu \re \frac{(\log p)\X(p)}{p} + iv \im\frac{(\log p)\X(p)}{p} \right)\right) + O\left(\frac{(|u|+|v|)\log p}{p^2}\right).$$
Therefore, if $p> \max(|u|\log |u|,|v|\log |v|)$ then 
\begin{equation}\label{Phip}
\begin{aligned}
\Phi_{\textup{rand}}(u, v; p) & =  \ex\Bigl(1+ iu \re \frac{(\log p)\X(p)}{p} + iv \im\frac{(\log p)\X(p)}{p}
 \\
& \hskip2cm - \frac12\Bigl(u \re \frac{(\log p)\X(p)}{p} + v \im\frac{(\log p)\X(p)}{p}\Bigr)^2 \Bigr)\\
& \quad + O\Bigl(\frac{(|u|+|v|)^3\log^3 p}{p^3}+ \frac{(|u|+|v|)\log p}{p^2}\Bigr)  \\
& = 1- (u^2 + v^2) \frac{\log^2p}{4p^2} + O\Bigl( \frac{(|u|+|v|)^3\log^3 p}{p^3}+ \frac{(|u|+|v|)\log p}{p^2}\Bigr),
\end{aligned}
\end{equation}
since $\ex(\X(p))=0$, $\ex\big(\re\X(p)\im\X(p)\big)=0$, and $\ex\big((\im\X(p)^2\big)= \ex\big((\re\X(p)^2\big)=1/2$.  
We now choose $X= A\max(|u|\log |u|,|v|\log |v|)$ for a suitably large constant $A>0$. Then inserting this estimate in \eqref{RangePrimes}, we obtain 
\begin{align*}
 |\Phi_{\textup{rand}}(u, v)| & \leq \exp\Bigl(- (u^2 + v^2) \sum_{p>X}\frac{\log^2p}{4p^2}\\
 & \hskip2cm + O\Bigl((|u|+|v|)^3 \sum_{p>X}\frac{\log^3 p}{p^3}+ (|u|+|v|)\sum_{p>X}\frac{\log p}{p^2}\Bigr)\Bigr)\\
 & \ll \exp\big(-c_0 (|u|+|v|)\big).
 \end{align*}
 where $c_0>0$ is a constant that depends on $A$. This completes the proof. 
\end{proof}
Since $\Phi_{\textup{rand}}(u, v)$ is exponentially decreasing by Proposition \ref{CharDecTrunc}, it follows from the Fourier inversion formula that the distribution of $\ld(1, \X)$ is absolutely continuous and has a smooth density function defined by 
$$g(x, y):= \frac{1}{(2\pi)^2}\int_{-\infty}^{\infty} \int_{-\infty}^{\infty} e^{-i(ux+vy)}\Phi_{\textup{rand}}(u, v) du dv .$$
To deduce Theorem \ref{Main} from Theorem \ref{Discrepancy}, we need to show that $g(0, 0)>0$. This follows from the following result of Borchsenius and Jessen \cite{BorJe}.
\begin{thm}[Borchsenius and Jessen \cite{BorJe}]\label{BorJeThm}
Let $\Y(n)$ be a sequence of independent random variables uniformly distributed on the unit circle. Let $f(z)=\sum_{k=1}^{\infty}\ell_k z^k$ be an analytic function in a disc $|z|<\rho$, such that $\ell_1\ne 0$. Let $\{r_n\}_{n\geq 1}$ and $\{\lambda_n\}_{n\geq 1}$ be sequences of real numbers such that $0<r_n< \rho$ and
$$  \sum_{n=1}^{\infty} |\lambda_n| r_n^2<\infty, \text{ and } \sum_{n=1}^{\infty} \lambda_n^2 r_n^2<\infty.
$$
Then the sum of random variables 
$$ \Y= \sum_{n=1}^{\infty} \lambda_nf(r_n\Y(n)),$$
is almost surely convergent and has a absolutely continuous distribution with a smooth density $h(x, y)$. Moreover, if $\sum_{n=1}^{\infty} |\lambda_n|r_n$ diverges then 
$h(x,y)>0$ for all $(x, y) \in \mathbb{R}^2$. 
\end{thm}
\begin{rem} Borchsenius and Jessen \cite{BorJe} only proved this result for the sum of random variables $\sum_{n=1}^{\infty} f(r_n\Y(n))$ (see Theorems 5 and 7 of \cite{BorJe}), but their proof extends easily to the more general case $\sum_{n=1}^{\infty} \lambda_nf(r_n\Y(n)).$
\end{rem}
\begin{cor} \label{BorJeCor} We have $g(x, y)>0$ for all $(x, y)\in \mathbb{R}^2$.
\end{cor}
\begin{proof} By \eqref{SumRandomInd} we have 
$$\ld(1, \X)= \sum_{p} (\log p) f\Bigl(\frac{\X(p)}{p}\Bigr), $$
where 
$$f(z)= \frac{z}{1-z}= \sum_{n=1}^{\infty} z^n,$$
is analytic in $|z|<1$. We can then verify that all the conditions of Theorem \ref{BorJeThm} are verified, since $\sum_p (\log p)/p^2$ and $\sum_p (\log p)^2/p^2$ converge, and $\sum_p (\log p)/p$ diverges. This completes the proof. 
\end{proof}

\begin{proof}[Deducing Theorem \ref{Main} from Theorem \ref{Discrepancy}] 
We recall that $q$ is a prime number.
 Let $\varepsilon=\varepsilon(q)>0$ be a small parameter to be chosen, such  that $\varepsilon(q) \to 0$ as $q\to \infty$. Let $\Psi_q(\varepsilon)$ denotes the number of non-principal characters $\chi\neq \chi_0 \bmod q $ such that
$$  
\Bigl|\frac{L^\prime}{L}(1, \chi)\Bigr| \leq \varepsilon.
$$
By Theorem \ref{Discrepancy} we have
\begin{equation}\label{LBSmallValues} 
\begin{aligned}
\frac{\Psi_q(\varepsilon)}{q-1} &\geq  \frac{1}{q-1}\Bigl|\Bigl\{\chi\neq \chi_0 \bmod q : \frac{L^\prime}{L}(1, \chi) \in 
\Bigl( -\frac{\varepsilon}{2}, \frac{\varepsilon}{2} \Bigr)^2 \Bigr\}\Bigr| \\
&= \pr\Bigl(\Lr\in \Bigl( -\frac{\varepsilon}{2}, \frac{\varepsilon}{2} \Bigr)^2 \Bigr) + O\Bigl(\frac{(\log\log q)^2}{\log q}\Bigr).
\end{aligned}
\end{equation}
On the other hand if $\varepsilon$ is suitably small then we have 
$$\pr\Bigl(\Lr\in \Bigl( -\frac{\varepsilon}{2}, \frac{\varepsilon}{2} \Bigr)^2 \Bigr)= \int_{-\varepsilon/2}^{\varepsilon/2} \int_{-\varepsilon/2}^{\varepsilon/2} g(x, y) dx dy \gg \varepsilon^2,$$
 since $g$ is continuous on $\mathbb{R}^2$ and $g(0, 0)>0$ by Corollary \ref{BorJeCor}. Hence, choosing $\varepsilon= C \log\log q/\sqrt{\log q}$ for some suitably large constant $C$ we deduce that 
$$ 
\Psi_q(\varepsilon) \gg  \frac{q(\log\log q)^2}{\log q},
$$
which implies the result.
\end{proof}

We end this section by proving the following the proposition, which gives uniform bounds for the moments of $|\ld(1, \X)|$. This will be used in the proof of Theorem \ref{Discrepancy}.
\begin{pro}\label{BoundMomRand} There exists a constant $c>0$ such that for all positive integers $k\geq 8$ we have 
$$ \ex\left(\left|\ld(1, \X)\right|^{2k}\right)\leq \big(c\log k\big)^{2k}.$$

\end{pro}
\begin{proof} Let $y>2$ be a real number to be chosen. By Minkowski's inequality and a weak form of the Prime Number Theorem we have 
\begin{equation}\label{Minkowski}
\begin{aligned}
\ex\left(\left|\ld(1, \X)\right|^{2k}\right)^{1/(2k)}
&\leq \ex\Bigl(\Bigl|\sum_{n\leq y}\frac{\Lambda(n)\X(n)}{n}\Bigr|^{2k}\Bigr)^{1/(2k)} +\ex\Bigl(\Bigl|\sum_{n>y}\frac{\Lambda(n)\X(n)}{n}\Bigr|^{2k}\Bigr)^{1/(2k)}\\
&\leq \sum_{n\leq y}\frac{\Lambda(n)}{n}+\ex\Bigl(\Bigl|\sum_{n>y}\frac{\Lambda(n)\X(n)}{n}\Bigr|^{2k}\Bigr)^{1/(2k)}
\\&
\ll \log y+\ex\Bigl(\Bigl|\sum_{n>y}\frac{\Lambda(n)\X(n)}{n}\Bigr|^{2k}\Bigr)^{1/(2k)}.\\
\end{aligned}
\end{equation}
Let $$\Lambda_{\ell,y}(n):=\sum_{\substack{n_1,n_2,\dots, n_{\ell}>y\\ n_1n_2\cdots n_{\ell}=n}} \Lambda(n_1)\Lambda(n_2)\cdots \Lambda(n_{\ell}).$$
Then, we have 
\begin{align*}
\ex\Bigl(\Big|\sum_{n>y}\frac{\Lambda(n)\X(n)}{n}\Bigr|^{2k}\Bigr)
&= \ex\Bigl(\sum_{n>y^k}\frac{\Lambda_{k,y}(n)\X(n)}{n}\sum_{n>y^k}\frac{\Lambda_{k,y}(m)\overline{\X(m)}}{m}\Bigr)\\
&=\sum_{n>y^{k}}\frac{\Lambda_{k,y}(n)^2}{n^2}\leq 
\sum_{n>y^{k}}\frac{(\log n)^{2k}}{n^2},
\end{align*}
since 
\begin{equation}\label{boundLambda}
\Lambda_{\ell,y}(n) \leq \Lambda_{\ell}(n) \leq \Bigl(\sum_{m|n}\Lambda(m)\Bigr)^{\ell}= (\log n)^{\ell}.
\end{equation}
Moreover, since $(\log n)^{2k}/\sqrt{n}$ is decreasing for $n>e^{4k}$, we deduce that if $y\geq e^4 $ then 
$$ \ex\Bigl(\Bigl|\sum_{n>y}\frac{\Lambda(n)X(n)}{n}\Bigr|^{2k}\Bigr)\leq \frac{(k\log y)^{2k}}{y^{k/2}}\sum_{n>y^{k}}\frac{1}{n^{3/2}}\ll\frac{(k\log y)^{2k}}{y^{k}}.$$
Choosing $y=k^2$ and inserting this estimate in \eqref{Minkowski} completes the proof.
\end{proof}


\section{Asymptotic formulas for the moments of $L^\prime/L(1, \chi)$ : Proof of Theorem \ref{MomentsLogDer}}\label{MomentsSec}

We first start with the following classical lemma, which provides a bound for $L^\prime/L(s,\chi)$ when $s$ is far from a zero of $L(z, \chi)$. 

\begin{lem}\label{BoundLD}
Let $\chi$ be a non-principal character modulo $q$. Let $t$ be a real number and suppose that $L(z,\chi)$ has no zeros for $\re(z)>\sigma_0$ and $|\im(z)|\leq |t|+1$. Then for any $\sigma>\sigma_0$ we have 
$$\frac{L^\prime}{L}(\sigma+it,\chi)\ll \frac{\log(q(|t|+2))}{\sigma-\sigma_0}.$$
\end{lem}
\begin{proof}
Let $\rho$ runs over the non-trivial zeros of $L(s,\chi)$. Then it follows from equation (4) of Chapter 16 of Davenport \cite{Da}
and a simple density theorem that 
\begin{align*}
\frac{L^\prime}{L}(\sigma+it,\chi)
&= \sum_{\substack{\rho \colon |t-\im(\rho)|<1}}\frac{1}{\sigma+it-\rho}+ O\big(\log(q(|t|+2))\big)\\
&\ll \frac{1}{\sigma-\sigma_0}\Bigl(\sum_{\substack{\rho\colon |t-\im(\rho)|<1}} 1\Bigr) + \log(q(|t|+2))\\
&\ll \frac{\log(q(|t|+2))}{\sigma-\sigma_0},
\end{align*}
as desired. 
\end{proof}
Using this lemma we can approximate large powers of $L^\prime/L(1,\chi)$ by short Dirichlet polynomials, if $L(s,\chi)$ has no zeros in a certain region to the left of the line $\re(s)=1$. 
\begin{pro}\label{ApproximationLarge} Let $0<\delta<1/2$ be fixed, and $q$ be large. Let $y\geq (\log q)^{10/\delta}$ be a  real number and $k\leq 2\log q/\log y$ be a positive integer. Then, for any non-principal character $\chi\bmod q$, if $L(s,\chi)$ is non-zero for $\re(s)>1-\delta$ and $|\im(s)|\leq y^{k\delta}$, then we have 
$$ \Bigl(\frac{L^\prime}{L}(1,\chi)\Bigr)^k=(-1)^k\sum_{n\leq y^k} \frac{\Lambda_k(n)}{n}\chi(n)+O_{\delta}\Big(y^{-k\delta/4}\Big),$$
where $\Lambda_k(n)$ is defined in \eqref{DefLambda}.
\end{pro}

\begin{proof}
Without loss of generality, suppose that $y^k\in \mathbb{Z}+1/2$. Let $c=1/(k\log y)$, and $T$ be a large real number to be chosen. Then by Perron's formula, we have 
$$ \frac{1}{2\pi i}\int_{c-iT}^{c+iT} \Bigl(\frac{L^\prime}{L}(1+s,\chi)\Bigr)^k \frac{y^{ks}}{s}ds= 
(-1)^k\sum_{n\leq y^k} \frac{\Lambda_k(n)}{n}\chi(n)+ O\Bigl(\frac{y^{kc}}{T}\sum_{n=1}^{\infty} \frac{\Lambda_k(n)}{n^{1+c}|\log(y^k/n)|}\Bigr).$$
To bound the error term of this last estimate, we split the sum into three parts: $n\leq y^k/2$, $y^k/2<n<2y^k$ and $n\geq 2y^k$. The terms in the first and third parts satisfy $|\log(y^k/n)|\geq \log 2$, and hence their contribution is 
$$\ll \frac{1}{T} \sum_{n=1}^{\infty}\frac{\Lambda_k(n)}{n^{1+c}}=\frac{1}{T} \Bigl(\sum_{n=1}^{\infty}\frac{\Lambda(n)}{n^{1+c}}\Bigr)^k\ll \frac{(2k\log y)^k}{T},$$
by the prime number theorem. To handle the contribution of the terms $y^k/2<n<2y^k$, we put $r=n-y^k$, and use that $|\log(y^k/n)|\gg |r|/y^k$. In this case, we have $\Lambda_k(n)\leq (\log n)^k\leq (2k\log y)^k$, and hence the contribution of these terms is 
$$\ll\frac{(2k\log y)^{k}}{Ty^{k}}\sum_{|r|\leq y^{k}}\frac{y^k}{|r|}\ll \frac{(2k\log y)^{k+1}}{T}.$$
We now choose $T=y^{k\delta/2}$ and move the contour to the line $\re(s)=-\delta/2$. By our assumption, we only encounter a simple pole at $s=0$ which leaves a residue $(L^\prime/L(1,\chi))^k$. Therefore, we deduce that 
$$ \frac{1}{2\pi i}\int_{c-iT}^{c+iT} \Bigl(\frac{L^\prime}{L}(1+s,\chi)\Bigr)^k \frac{y^{ks}}{s}ds= \Bigl(\frac{L^\prime}{L}(s,\chi)\Bigr)^k+ E_1,$$
where 
\begin{align*}
E_1&=\frac{1}{2\pi i} \left(\int_{c-iT}^{-\delta/2-iT}+ \int_{-\delta/2-iT}^{-\delta/2+iT}+ \int_{-\delta/2+iT}^{c+iT}\right) \Bigl(-\frac{L^\prime}{L}(1+s,\chi)\Bigr)^k \frac{y^{ks}}{s}ds\\
& \ll_{\delta} \frac{(\log (qT))^k}{T}+ y^{-k\delta/2}\Bigl(\frac{\log (qT)}{\delta}\Bigr)^{k+1}\\
&\ll_{\delta} y^{-k\delta/4},
\end{align*}
by Lemma \ref{BoundLD}. Finally, since $(2k\log y)^{k+1}/T\ll y^{-k\delta/4}$, the result follows.  
\end{proof}

Now, using a standard zero density estimate due to Montgomery (see equation \eqref{ZeroDensity} below), we deduce from Proposition \ref{ApproximationLarge} that large powers of $L^\prime/L(1,\chi)$ can be approximated by short Dirichlet polynomials for almost all non-principal characters $\chi \bmod q$.
\begin{cor}\label{AAApproximation}
Let $q$ be a large prime. Let $k$ be a positive integer such that $k\leq \log q/(50\log\log q)$.  For all except $O(q^{3/4})$ non-principal characters $\chi \bmod q$ we have 
$$ \Bigl(\frac{L^\prime}{L}(1,\chi)\Bigr)^k=(-1)^k\sum_{n\leq q} \frac{\Lambda_k(n)}{n}\chi(n)+O\Big(q^{-1/20}\Big).$$
\end{cor}

\begin{proof} Let $N(\sigma, T, \chi)$ denote the number of zeros of $L(s, \chi)$ in the rectangle $ \sigma<\re(s)\leq 1$ and $|\im(s)|\leq T$. The standard zero density result of Montgomery \cite{Mo} states that for $q,T\geq 2$ and $1/2\leq \sigma\leq 4/5$ we have 
\begin{equation}\label{ZeroDensity}
 \sum_{\chi \bmod q} N(\sigma, T, \chi)\ll (qT)^{3(1-\sigma)/(2-\sigma)} (\log(qT))^9.
\end{equation}
Choosing $\delta=1/5$, we deduce that for all except $O(q^{3/4})$ non-principal characters $\chi \bmod q$, $L(s, \chi)$ does not vanish in the region $\re(s)>1-\delta$ and $|\im(s)|\leq q^{\delta}$. We now take $y=q^{1/k}$  in Proposition \ref{ApproximationLarge}, to obtain that for all except $O(q^{3/4})$ non-principal characters $\chi \bmod q$ we have 
$$ \Bigl(\frac{L^\prime}{L}(1,\chi)\Bigr)^k=(-1)^k\sum_{n\leq q} \frac{\Lambda_k(n)}{n}\chi(n)+O\Big(q^{-1/20}\Big),$$
as desired.
\end{proof}
Another consequence of Proposition \ref{ApproximationLarge} is that $L^\prime/L(1, \chi)\ll \log\log q$ for all except for a small exceptional set of  non-principal characters $\chi \bmod q $.

\begin{cor}\label{ASBound}
Let $q$ be a large prime. Then for all but $O(q^{3/4})$ non-principal characters $\chi \bmod q$ we have
$$ \frac{L^\prime}{L}(1, \chi)\ll \log\log q.$$
\end{cor}
\begin{proof}
Taking $\delta=1/5$, $k=1$ and $y=(\log q)^{50}$ in Proposition \ref{ApproximationLarge} and using \eqref{ZeroDensity} as in the proof of Corollary \ref{AAApproximation} we deduce that for all except $O(q^{3/4})$ non-principal characters $\chi \bmod q$, we have 
\[
\frac{L^\prime}{L}(1,\chi)=-\sum_{n\leq y} \frac{\Lambda(n)}{n}\chi(n)+O\Big(y^{-1/20}\Big)\ll \log\log q.
\tag*{\qed}
\]
\renewcommand{\qed}{}
\end{proof}

We now have all the ingredients  to establish asymptotic formulas for large moments of $L^\prime/L(1, \chi)$, over the characters 
$\chi\in \F_q$, where $\F_q$ is defined in \eqref{Fq-def}.

\begin{proof}[Proof of Theorem \ref{MomentsLogDer}]
First, note that for any positive integer $r\geq 1$  by the Prime Number Theorem we have 
\begin{equation}\label{BoundPowerLambda}
\sum_{n\leq q}\frac{\Lambda_r(n)}{n} \leq  \Bigl(\sum_{n\leq q}\frac{\Lambda(n)}{n}\Bigr)^r \ll (2\log q)^r.
\end{equation}
Let $\E_q$ be the exceptional set in Corollary \ref{AAApproximation}. Then it follows from this result that
\begin{equation}\label{ApproxShortMoments}
\begin{aligned}
& \frac{1}{q-1}\sum_{\chi\in \F_q\setminus \E_q} \Bigl(\frac{L^\prime}{L}(1,\chi)\Bigr)^k\Bigl(\overline{\frac{L^\prime}{L}(1,\chi)}\Bigr)^{\ell}\\
& = \frac{(-1)^{k+\ell}}{q-1}\sum_{\chi \in \F_q\setminus \E_q} \sum_{n\leq q} \frac{\Lambda_k(n)}{n}\chi(n)\sum_{m\leq q} \frac{\Lambda_{\ell}(m)}{m}\overline{\chi(m)}+E_2,
\end{aligned}
\end{equation}
where 
$$ E_2 \ll q^{-1/20} (2\log q)^{\max(k,\ell)} \ll q^{-1/30}, $$
by \eqref{BoundPowerLambda}.  Now, by \eqref{BoundPowerLambda} and the orthogonality of Dirichlet characters the main term on the right hand side of \eqref{ApproxShortMoments} equals 
\begin{align*}
&(-1)^{k+\ell}\sum_{n\leq q} \frac{\Lambda_k(n)}{n}\sum_{m\leq q} \frac{\Lambda_{\ell}(m)}{m}\frac{1}{q-1}\sum_{\chi \in \F_q\setminus \E_q} \chi(n)\overline{\chi(m)}\\
&= (-1)^{k+\ell}\sum_{n, m\leq q} \frac{\Lambda_k(n)\Lambda_{\ell}(m)}{nm} \frac{1}{q-1} 
\sum_{\chi \bmod q} \chi(n)\overline{\chi(m)} +
 O\left(\frac{(2\log q)^{k+\ell}}{q^{1/4}}\right)\\
 & = (-1)^{k+\ell} \sum_{n\leq q} \frac{\Lambda_k(n)\Lambda_{\ell}(n)}{n^2} + O\left(q^{-1/8}\right),
 \end{align*}
 since $|\E_q|\ll q^{3/4}$.
Finally using \eqref{boundLambda}, together with the fact that the function $(\log t)^k/\sqrt{t}$ is decreasing for $t\geq e^{2k}$, we obtain
$$ \sum_{n> q}
\frac{\Lambda_{k}(n)\Lambda_{\ell}(n)}{n^2}\leq \sum_{n>q}
\frac{(\log n)^{k+\ell}}{n^2} \ll \frac{(\log q)^{k+\ell}}{\sqrt{q}}\sum_{n>q}\frac{1}{m^{3/2}}\ll \frac{(\log q)^{k+\ell}}{q}\ll q^{-1/2}.$$
Inserting these estimates in \eqref{ApproxShortMoments} gives 

\begin{equation}\label{AsympMoments1}
\frac{1}{q-1}\sum_{\chi\in \F_q\setminus \E_q} \Bigl(\frac{L^\prime}{L}(1,\chi)\Bigr)^k\Bigl(\overline{\frac{L^\prime}{L}(1,\chi)}\Bigr)^{\ell}= (-1)^{k+\ell}\sum_{n=1}^{\infty} \frac{\Lambda_k(n) \Lambda_{\ell}(n)}{n^2}+ O\left(q^{-1/30}\right). 
\end{equation}
Furthermore, it follows from Lemma \ref{BoundLD} along with the classical zero-free region for $L(s, \chi)$ that for $\chi \in \F_q$ we have 
\begin{equation}\label{ClassicalBound}
\frac{L^\prime}{L}(1,\chi)\ll (\log q)^2.
\end{equation}
Therefore, combining this bound with \eqref{AsympMoments1} yields
\begin{align*}
&\frac{1}{q-1}\sum_{\chi\in \F_q} \Bigl(\frac{L^\prime}{L}(1,\chi)\Bigr)^k\Bigl(\overline{\frac{L^\prime}{L}(1,\chi)}\Bigr)^{\ell}\\
 & =\frac{1}{q-1}\sum_{\chi\in \F_q\setminus \E_q} \Bigl(\frac{L^\prime}{L}(1,\chi)\Bigr)^k\Bigl(\overline{\frac{L^\prime}{L}(1,\chi)}\Bigr)^{\ell} +O\left(q^{-1/4}(\log q)^{2k+2\ell}\right)\\
& =(-1)^{k+\ell}\sum_{n=1}^{\infty} \frac{\Lambda_k(n) \Lambda_{\ell}(n)}{n^2}+ O\left(q^{-1/30}\right). 
\tag*{\qed}
\end{align*}
\renewcommand{\qed}{}
\end{proof}


\section{Bounding the discrepancy of the distribution of  $L^\prime/L(1, \chi)$: Proof of Theorem \ref{Discrepancy}} \label{DiscrepancySec}
Theorem \ref{Discrepancy} is proved along the same lines of Theorem 1.1 of \cite{LLR}, which bounds the discrepancy of the distribution of the logarithm of the Riemann zeta function to the right of the critical line. The main ingredient of the proof is the following result, which shows that the characteristic function of the joint distribution of   $\re (L^\prime/L(1, \chi))$ and $\im (L^\prime/L(1, \chi))$ is very close to that of the random variables $\re (\Lr)$ and $\im(\Lr)$. 
For $u,v\in \mathbb{R}$ we define
$$ 
\Phi_{q}(u,v):= \frac{1}{q-1} \sum_{\chi\neq \chi_0 \bmod q}
\exp\Bigl(i u \re\frac{L^\prime}{L}(1, \chi)+iv \im\frac{L^\prime}{L}(1, \chi)\Bigr).
$$
Then we prove
\begin{thm}\label{Characteristic}
Let $q$ be a large prime. There exists an absolute constant $b_0>0$
such that for $|u|,|v| \leq b_0(\log q)/(\log \log q)^2$ we have
$$
\Phi_{q}(u,v)=\Phi_{\textup{rand}}(u,v)+O\Bigl(\exp\Bigl(-\frac{\log q}{100\log\log q}\Bigr)\Bigr),
$$
where $\Phi_{\textup{rand}}(u, v)$ is defined \eqref{DefPhiRand}.
\end{thm}
\begin{proof}
Let $N= \lfloor \log q/(100\log\log q)\rfloor$ and put $r= \max(|u|, |v|)$.  Recalling \eqref{Fq-def} and using the Taylor expansion of $e^u$, we have
\begin{equation}\label{Char1}
\begin{aligned}
\Phi_{q}(u,v)
&= \frac{1}{q-1} \sum_{\chi\in \F_q}
\exp\Bigl(i u \re\frac{L^\prime}{L}(1, \chi)+iv \im\frac{L^\prime}{L}(1, \chi)\Bigr) +O\Bigl(\frac{1}{q}\Bigr)\\
&= \sum_{n\leq 2N} \frac{i^n}{n!} \frac{1}{q-1}\sum_{\chi\in \F_q}\Bigl(u \re\frac{L^\prime}{L}(1, \chi)+v \im\frac{L^\prime}{L}(1, \chi)\Bigr)^n \\
& \quad \quad + 
O\Bigl(\frac{(2r)^{2N}}{(2N)!}\frac{1}{q-1}\sum_{\chi\in \F_q}\Bigl|\frac{L^\prime}{L}(1, \chi)\Bigr|^{2N}+\frac1q\Bigr).
\end{aligned}
\end{equation}
Now, by Theorem \ref{MomentsLogDer} and Proposition \ref{BoundMomRand} we get 
\begin{equation}\label{UpperMomentsRand2}
\frac{1}{q-1}\sum_{\chi\in \F_q}\Bigl|\frac{L^\prime}{L}(1, \chi)\Bigr|^{2N}= 
\ex\left(\left|\ld(1, \X)\right|^{2N}\right) +O\left(q^{-1/30}\right) \ll (c\log N)^{2N}.
\end{equation}
Hence, by Stirling's formula the error term of \eqref{Char1} is 
$$ \ll \Bigl(\frac{3c r\log N}{N}\Bigr)^{2N} +\frac{1}{q} \ll e^{-N},$$
by our assumption on $u$ and $v$, if $b_0$ is small enough. 

Now, let $z_1= (u-iv)/2$ and $z_2= (u+iv)/2$. Then, it follows from Theorem \ref{MomentsLogDer} that the inner sum in the main term of \eqref{Char1} equals
\begin{equation}\label{Char2} 
\begin{aligned}
&\frac{1}{q-1}\sum_{\chi\in \F_q}\Bigl(u \re\frac{L^\prime}{L}(1, \chi)+v \im\frac{L^\prime}{L}(1, \chi)\Bigr)^n \\
&= \frac{1}{q-1}\sum_{\chi\in \F_q}\Bigl(z_1\frac{L^\prime}{L}(1, \chi)+z_2 \overline{\frac{L^\prime}{L}(1, \chi)}\Bigr)^n\\
& = \sum_{j=0}^n \binom{n}{j} z_1^j z_2^{n-j} \frac{1}{q-1}\sum_{\chi\in \F_q} \Bigl(\frac{L^\prime}{L}(1, \chi)\Bigr)^j  \Bigl(\overline{\frac{L^\prime}{L}(1, \chi)}\Bigr)^{n-j}\\
& = \sum_{j=0}^n \binom{n}{j} z_1^j z_2^{n-j} \ex\left(\ld(1, \X)^j \overline{\ld(1, \X)}^{n-j}\right) + O\left((2r)^nq^{-1/30} \right)\\
& = \ex\Big(\big(u \re\ld(1, \X)+v \im\ld(1, \X)\big)^n\Big) + O\left((2r)^nq^{-1/30} \right).
\end{aligned}
\end{equation}
Now, repeating the same argument leading to \eqref{Char1} but for the random model $\ld(1, \X)$, and using the bound \eqref{UpperMomentsRand2}  we deduce that 
$$
\Phi_{\textup{rand}}(u,v)= \sum_{n\leq 2N} \frac{i^n}{n!}\ex\Big(\big(u \re\ld(1, \X)+v \im\ld(1, \X)\big)^n\Big) + O\bigl(e^{-N}\bigr). 
$$ 
Finally, combining this estimate with \eqref{Char1} and \eqref{Char2} completes the proof.
\end{proof}

To deduce Theorem \ref{Discrepancy} from Theorem \ref{Characteristic} we use Beurling-Selberg functions.
For $z\in \mathbb C$ let
\[
H(z) =\Bigl( \frac{\sin \pi z}{\pi} \Bigr)^2 \Bigl( \sum_{n=-\infty}^{\infty} \frac{\textup{sgn}(n)}{(z-n)^2}+\frac{2}{z}\Bigr)
\qquad\mbox{and} \qquad K(z)=\Big(\frac{\sin \pi z}{\pi z}\Big)^2.
\]
Beurling proved that the function $B^+(x)=H(x)+K(x)$
majorizes $\textup{sgn}(x)$ and its Fourier transform
has restricted support in $(-1,1)$. Similarly, the function $B^-(x)=H(x)-K(x)$ minorizes $\textup{sgn}(x)$ and its Fourier
transform has the same property (see Lemma 5 of Vaaler \cite{Va}).

Let $\Delta>0$ and $a,b$ be real numbers with $a<b$. Take $\mathcal I=[a,b]$
and
define
\[
F_{\mathcal{I}, \Delta} (z)=\frac12 \Big(B^-(\Delta(z-a))+B^-(\Delta(b-z))\Big).
\]
Then we have the following lemma, which is proved in \cite{LLR} (see Lemma 4.7 therein and the discussion above it).

\begin{lem} \label{lem:functionbd}
The function $F_{\mathcal{I}, \Delta}$ satisfies the following properties
\begin{itemize}
\item[1.]
For all $x\in \mathbb{R}$ we have $
|F_{\mathcal{I}, \Delta}(x)| \le 1
$ and  
\begin{equation} \label{l1 bd}
0 \le \mathbf 1_{\mathcal I}(x)- F_{\mathcal{I}, \Delta}(x)\le K(\Delta(x-a))+K(\Delta(b-x)).
\end{equation}
\item[2.]
 The Fourier transform of $F_{\mathcal{I}, \Delta}$ is
\begin{equation} \label{Fourier}
\widehat F_{\mathcal{I}, \Delta}(\xi)=
\begin{cases}\widehat{ \mathbf 1}_{\mathcal I}(\xi)+O\big(1/\Delta \big) & \mbox{ if } |\xi| < \Delta, \\
0 & \mbox{ if } |\xi|\ge \Delta.
\end{cases}
\end{equation}

\end{itemize}
\end{lem}

\begin{proof}[Proof of Theorem \ref{Discrepancy}]
First, Corollary \ref{ASBound} shows that it suffices to consider rectangles $\mathcal{R}$ contained in $[-(\log\log q)^2, (\log\log q)^2]^{2}.$ Let $\mathcal{R}=[a,b] \times[c,d]$, with
$|b-a|, |c-d|\le 2(\log\log q)^2$. We also write $\mathcal I=[a,b]$ and $\mathcal J=[c,d]$. 

Let $\Delta=b_0(\log q)/(\log \log q)^2$ where $b_0$ is the corresponding constant in Theorem \ref{Characteristic}. By Fourier inversion, \eqref{Fourier}, and Theorem \ref{Characteristic} we have that
\begin{equation} \label{long est}
\begin{split}
&\frac{1}{q-1} \sum_{\chi\neq \chi_0}  F_{\mathcal I, \Delta} \Bigl(\textup{Re}\frac{L^\prime}{L}(1, \chi)\Bigr) 
F_{\mathcal J, \Delta}\Bigl(\textup{Im}\frac{L^\prime}{L}(1, \chi)\Bigr)\\
&
=\frac{1}{(2\pi)^2}\int_{-\infty}^{\infty}\int_{-\infty}^{\infty}  \widehat{F}_{\mathcal I, \Delta} (u) 
\widehat{F}_{\mathcal J, \Delta}(v)  \Phi_q(u, v) \, d u \, d v \\
&
= \frac{1}{(2\pi)^2}\int_{-\Delta}^{\Delta}\int_{-\Delta}^{\Delta}  \widehat{F}_{\mathcal I, \Delta} (u) 
\widehat{F}_{\mathcal J, \Delta}(v)  \Phi_{\textup{rand}}(u, v) \, d u \, d v +O\Bigl(\frac{\big(\Delta(\log\log q)^2\big)^{2}}{(\log q)^{10}}\Bigr)\\
& 
=\mathbb E \Bigl(F_{\mathcal I, \Delta} \big(\textup{Re}\ld(1, \X)\bigr) 
F_{\mathcal J, \Delta}\big( \textup{Im} \ld(1, \X)\big) \Bigr)+O\Bigl(\frac{1}{(\log q)^2}\Bigr).
\end{split}
\end{equation}

Next note that $\widehat K(\xi)=\max(0,1-|\xi|)$. Applying Fourier inversion, Theorem \ref{Characteristic},
and Proposition \ref{CharDecTrunc} we obtain
\begin{equation} \notag
\begin{split}
& \frac{1}{q-1} \sum_{\chi\neq \chi_0}  K\Big( \Delta \cdot \Big(\textup{Re} \frac{L^\prime}{L}(1, \chi)-\alpha\Big)\Big)
=\frac{1}{2\pi\Delta}\int_{-\Delta}^{\Delta}\Big(1-\frac{|\xi|}{\Delta}\Big) e^{-i \alpha \xi} \Phi_q(\xi,0) \, d\xi \ll  \frac{1}{\Delta},
\end{split}
\end{equation}
where $\alpha$ is an arbitrary real number. By this and \eqref{l1 bd} we get that
\begin{equation} \label{K bd}
\frac{1}{q-1} \sum_{\chi\neq \chi_0}  F_{\mathcal I, \Delta}\Bigl(\textup{Re} \frac{L^\prime}{L}(1, \chi)\Bigr) 
= \frac{1}{q-1} \sum_{\chi\neq \chi_0} \mathbf 1_{\mathcal I}\Bigl(\textup{Re} \frac{L^\prime}{L}(1, \chi)\Bigr)+O\Bigl(\frac{1}{\Delta}\Bigr).
\end{equation}
Lemma \ref{lem:functionbd} implies that $|F_{\mathcal J, \Delta}(x)| \le 1$ and $ \mathbf 1_{\mathcal I}(x)- F_{\mathcal{I}, \Delta}(x)\geq 0$. Hence, by this and \eqref{K bd} we derive
\begin{equation} \label{IndicatorApprox}
\begin{split}
&\frac{1}{q-1} \sum_{\chi\neq \chi_0}  F_{\mathcal I, \Delta} \Bigl(\textup{Re}\frac{L^\prime}{L}(1, \chi)\Bigr) 
F_{\mathcal J, \Delta}\Bigl(\textup{Im}\frac{L^\prime}{L}(1, \chi)\Bigr) \\
&=\frac{1}{q-1} \sum_{\chi\neq \chi_0}  \mathbf 1_{\mathcal I} \Bigl(\textup{Re}\frac{L^\prime}{L}(1, \chi)\Bigr) 
F_{\mathcal J, \Delta}\Bigl(\textup{Im}\frac{L^\prime}{L}(1, \chi)\Bigr)+O\Bigl(\frac{1}{\Delta}\Bigr).
\end{split}
\end{equation}
A similar argument leading to \eqref{K bd} yields
$$
\frac{1}{q-1} \sum_{\chi\neq \chi_0}  F_{\mathcal J, \Delta}\Bigl(\textup{Im} \frac{L^\prime}{L}(1, \chi)\Bigr) 
= \frac{1}{q-1} \sum_{\chi\neq \chi_0} \mathbf 1_{\mathcal J}\Bigl(\textup{Im} \frac{L^\prime}{L}(1, \chi)\Bigr)+O\Bigl(\frac{1}{\Delta}\Bigr).
$$
Hence, combining this estimate with \eqref{IndicatorApprox} and using Lemma \ref{lem:functionbd} we obtain
\begin{equation} \label{one}
\begin{split}
&\frac{1}{q-1} \sum_{\chi\neq \chi_0}  F_{\mathcal I, \Delta} \Bigl(\textup{Re}\frac{L^\prime}{L}(1, \chi)\Bigr) 
F_{\mathcal J, \Delta}\Bigl(\textup{Im}\frac{L^\prime}{L}(1, \chi)\Bigr)\\
& = \frac{1}{q-1} \Bigl|\Bigl\{\chi\neq \chi_0 \bmod q : \frac{L^\prime}{L}(1, \chi) \in \mathcal{R} \Bigr\}\Bigr|+ O\Bigl(\frac{1}{\Delta}\Bigr).
\end{split}
\end{equation}
A similar argument applied to the random model shows that 
\begin{equation} \label{two}
\mathbb E \Bigl(F_{\mathcal I, \Delta} \big(\textup{Re}\ld(1, \X)\big) 
F_{\mathcal J, \Delta}\big( \textup{Im} \ld(1, \X)\big) \Bigr)= \pr\left(\ld(1, \X) \in \mathcal{R}\right)+ O\Bigl(\frac{1}{\Delta}\Bigr).
\end{equation}
Inserting the estimates \eqref{one} and \eqref{two} in \eqref{long est} completes the proof.
\end{proof}


\section{Computational part; proof of Theorem \ref{comput-mq}  and Corollary \ref{corollary-nonzero}}
\label{CompSec}
Recalling the main definitions in  Section 3 of  \cite{Languasco2019},
we denote the first $\chi$-Bernoulli number as
\(
B_{1,\chi}
: = 
q^{-1}\sum_{a=1}^{q-1}  a \chi(a) ,
 \)
 and  $R(x)= - \frac{\partial^2}{\partial s^2} \zeta(s,x) \vert _{s=0}
=\log (\Gamma_1(x))$, $x>0$, where $\zeta(s,x)$ is the Hurwitz zeta-function, and 
$s\in \C\setminus\{1\}$.
By eq.~(3.5)-(3.6) of Deninger \cite{Deninger1984} we have
 \begin{align*}
R(x) &:=  -\zeta^{\prime\prime}(0) - S(x),\\
S(x) 
&:=
2 \gamma_1 x +(\log x)^{2} 
+ 
    \sum_{m=1}^{+\infty}
\Bigl(
\bigl(\log (x+m)\bigr)^{2} - (\log m)^{2} -2x \frac{\log m}{m}
\Bigr),
\end{align*}
where
 \begin{equation*} 
\gamma_{1}  =
\lim_{N\to+\infty}
\Bigl(
\sum_{j=1}^{N}
\frac{\log j}{j} - \frac{(\log N)^2}{2}
\Bigr) ,
\quad
%
%
\zeta^{\prime\prime}(0)= \frac12 \Bigl(-(\log 2\pi)^2-\frac{\pi^2}{12} + \gamma_1 + \gamma^2 \Bigr)
\end{equation*}
and  $\gamma$ is the Euler-Mascheroni constant.
Arguing as in Sections 3.1-3.2 of  \cite{Languasco2019}
we have that
\begin{equation} 
\label{mq-odd-def}
m^{\textrm{odd}}_q
:=\min_{\chi\, \textrm{odd}} \
\Bigl\vert\frac{L^\prime}{L} (1,\chi) \Bigr\vert  
= 
\min_{\chi\, \textrm{odd}} \
 \Bigl\vert
\gamma + \log(2\pi)  
 +
\frac{1}{B_{1,\overline{\chi}} }
 \sum_{a=1}^{q-1}   \overline{\chi}(a)\log\Bigl(\Gamma\bigl(\frac{a}{q}\bigr)\Bigr)
 \Bigr\vert 
\end{equation}  
and
\begin{equation}
\label{mq-even-def} 
m^{\textrm{even}}_q
:=\min_{\substack{\chi \neq \chi_0\\ \chi\, \textrm{even}}} \
\Bigl\vert\frac{L^\prime}{L} (1,\chi) \Bigr\vert  
=
\min_{\substack{\chi \neq \chi_0\\ \chi\, \textrm{even}}} \
 \Bigl\vert
\gamma + \log(2\pi) 
-\frac{1}{2} 
\frac{\sum_{a=1}^{q-1} \overline{\chi}(a) \ S(a/q)}
{\sum_{a=1}^{q-1} \overline{\chi}(a)\log\bigl(\Gamma(a/q)\bigr)}
 \Bigr\vert ,
\end{equation}
where  $\Gamma$ is Euler's function.
In this way  we can compute $m_q= \min(m^{\textrm{odd}}_q, m^{\textrm{even}}_q)$, $3\le q \le \bound$, $q$ prime,
using the values of $\log\Gamma$ and $S$ obtained in \cite{Languasco2019} for $q\le 10^6$ and with
a new set of computation for $10^6< q\le \bound$.
We recall  that computing
the needed values of $S(a/q)$  is the most time consuming step of the whole procedure; for a detailed description
on how to obtain such values, we  refer to \cite{Languasco2019} and to a new,  much faster, algorithm
developed by Languasco and Righi in \cite{LanguascoR2020}.
In fact it is the latter method that made it possible to obtain the new set of results for $10^6< q\le \bound$.

\subsection{Computations using  PARI/GP (slower but with more digits available).}
First of all we notice that PARI/GP, v.~2.11.4, has the ability to  generate
 the Dirichlet $L$-functions (and  many other $L$-functions)
and hence the computation of $m_q$  can be  performed  directly using \eqref{mq-def}
with a linear cost in the number of calls
of the {\tt lfun} function of PARI/GP. This, at least  
on our Dell OptiPlex-3050 desktop machine
 (Intel i5-7500 processor, 3.40GHz, 
16 GB of RAM and running Ubuntu 18.04.2),
is slower than  using \eqref{mq-odd-def}-\eqref{mq-even-def}.
Using such equations, we wrote a suitable gp script
to obtain the  values of  $m_q$  for every odd prime $q$ up to $1000$  with 
a precision of $30$ digits, see  Table \ref{table1}.

\subsection{Computations using the C programming language and the fftw software library}
For larger values of $q$ we exploited the Fast Fourier Transform approach 
described in Sections 4.1-4.2-4.3 of \cite{Languasco2019}; in this 
case we used the \texttt{fftw} software library. This method
is much faster than the one described in the previous paragraph, but produces less digits in the numerical values obtained for $m_q$.
Using the data on $S(a/q)$ for every odd prime $q\le 10^6$
and $a=1,\dotsc,q-1$ obtained in 
\cite{Languasco2019} and, for $10^6<q\le \bound$, computed 
with the algorithm described in \cite{LanguascoR2020}, we obtained
the values of $m_q$  for every odd prime $q\le \bound$ using the \emph{long double precision} (80 bits)
of the C programming language.  For $q$ up to $10^6$ this just required about one day 
of  time on the Dell Optiplex machine previously mentioned since the data on $S(a/q)$
were already available from \cite{Languasco2019}.
For the remaining $q$-range the new current computation was performed on the 
University of Padova Strategic Research Infrastructure 
Grant 2017: ``CAPRI: Calclo ad Alte Prestazioni per la Ricerca e l'Innovazione'',
\url{http://capri.dei.unipd.it};
this step was performed using at most  60 computing nodes and it required about 48 hours of  time
(the global execution time, obtained by summing the declared computing time on each node, was of 101 days and 6 hours).
Moreover, we recomputed, using a \emph{quadruple precision} (128 bits) version
of our program, the $192$ cases in which we have $m_q<10^{-5}$. 
The total time required for performing such verifications  was about twelve hours
(on the same Optiplex machine already mentioned).
We also remark  that $m^{\textrm{odd}}_q > m^{\textrm{even}}_q$ for $333408$ 
cases  over  a total number of primes
equal to $664578$ ($50.17\%$)
and that $m^{\textrm{even}}_q > m^{\textrm{odd}}_q$ in the remaining $331170$ cases  ($49.83\%$). 

 The  minimal value is $(6311157483\dotsc)\cdot10^{-7}$
 and it is attained at $q=6119053$; 
 the  maximal value is $0.3682816159701500\dotsc$ and it is attained at $q=3$. 

\subsection{Proofs of Theorem \ref{comput-mq} and Corollary \ref{corollary-nonzero}}
An analysis on the  data computed in the previous 
subsections reveals that
\begin{equation}
\label{mq-comput-bounds}
\frac{21}{200q}< m_q < \frac{5}{\sqrt{q}}
\end{equation}
for every odd prime $3 \le q\le \bound$.   This proves Theorem \ref{comput-mq}. Such $m_q$-values
are collected in a comma-separated values (csv) file available, together
with the programs that performed the analysis leading to equation \eqref{mq-comput-bounds},
 at the following web address:
\url{http://www.math.unipd.it/~languasc/smallvalues/results}.  
In Section \ref{tables-figures} we include some scatter plots for the normalised values of $m_q^\prime :=\frac{200} {21}q m_q$
to visualise the truth of the lower bound in \eqref{mq-comput-bounds}.
The plots were obtained  using GNUPLOT, v.5.2, patchlevel 8. 

As a consequence of \eqref{mq-comput-bounds} we have that   
$L^\prime(1,\chi) \ne 0$, for every non trivial primitive Dirichlet character  $\bmod\ q$, 
$3\le q \le \bound$, $q$ prime. Hence Corollary \ref{corollary-nonzero} is proved.

\renewcommand{\bibliofont}{\normalsize}

\vskip0.5cm
{\sc
\noindent
Youness Lamzouri\\
Institut \'Elie Cartan de Lorraine, Universit\'e de Lorraine, BP 70239, 
54506 Vandoeuvre-l\`es-Nancy Cedex, France; 
and Department of Mathematics and Statistics,
York University,
4700 Keele Street,
Toronto, ON,
M3J1P3
Canada.\\
}
\emph{e-mail address}: youness.lamzouri@univ-lorraine.fr

\vskip0.5cm
{\sc
\noindent
 Alessandro Languasco\\
 Universit\`a di Padova,
 Dipartimento di Matematica,
 ``Tullio Levi-Civita'',
Via Trieste 63,
35121 Padova, Italy.
}
\emph{e-mail address}: alessandro.languasco@unipd.it

\newpage
\section{Tables and Figures}
\label{tables-figures}
\begin{table}[htp]
\scalebox{0.5875}{
\begin{tabular}{|c|c|}
\hline
$q$  &  $m_{q}$\\ \hline
$3$ & $0.368281615970147842633237904076\dotsc$ \\
$5$ & $0.180899098585657908884214228728\dotsc$ \\
$7$  & $0.015635689993720378956622751350\dotsc$ \\
$11$ & $0.084218304297040925093687383995\dotsc$ \\
$13$ & $0.300105391273262471564455827946\dotsc$ \\
$17$ & $0.215168738351581113325995469061\dotsc$ \\
$19$ & $0.084913681787711588506979393826\dotsc$ \\
$23$ & $0.222264564054341426307285821914\dotsc$ \\
$29$ & $0.186466418002260383262831736558\dotsc$ \\
$31$ & $0.156365159195612900544888732701\dotsc$ \\
$37$ & $0.084582297210773917089404219321\dotsc$ \\
$41$ & $0.038491048531500073439045451195\dotsc$ \\
$43$ & $0.137995302770293343459953078899\dotsc$ \\
$47$ & $0.035746012624318111324062775199\dotsc$ \\
$53$ & $0.079345452636605470144626619119\dotsc$ \\
$59$ & $0.070814808482221803352134970016\dotsc$ \\
$61$ & $0.004424742139200355181771341999\dotsc$ \\
$67$ & $0.101724238410284799512624760672\dotsc$ \\
$71$ & $0.083677019184249846969185188938\dotsc$ \\
$73$ & $0.037814195629563525097422478059\dotsc$ \\
$79$ & $0.066716629702353438341139676134\dotsc$ \\
$83$ & $0.106137806076174129263066902611\dotsc$ \\
$89$ & $0.091454122541715140553245678638\dotsc$ \\
$97$ & $0.100225166851282154186793751382\dotsc$ \\
$101$ & $0.088532955088052167463294100498\dotsc$ \\
$103$ & $0.067089590517766187761945182653\dotsc$ \\
$107$ & $0.072842375116831038918164998747\dotsc$ \\
$109$ & $0.050769692347897029380594369802\dotsc$ \\
$113$ & $0.137803959714882644119589746949\dotsc$ \\
$127$ & $0.040736836472454861641890942261\dotsc$ \\
$131$ & $0.034839900905728525833611106595\dotsc$ \\
$137$ & $0.173847605183310967783545356651\dotsc$ \\
$139$ & $0.064875376939531262338130695938\dotsc$ \\
$149$ & $0.015584848820092211525156310710\dotsc$ \\
$151$ & $0.076491986002214823571286680298\dotsc$ \\
$157$ & $0.089056313036529923958481456478\dotsc$ \\
$163$ & $0.007515153792183843621864583202\dotsc$ \\
$167$ & $0.031447516352414412954702343316\dotsc$ \\
$173$ & $0.045901456916810041131888422282\dotsc$ \\
$179$ & $0.079209917602758559902760698720\dotsc$ \\
$181$ & $0.083976658136399152061230909218\dotsc$ \\
$191$ & $0.078930822278753140729685017610\dotsc$ \\
$193$ & $0.070704014179180612385369954742\dotsc$ \\
$197$ & $0.044199502798542042544024860233\dotsc$ \\
$199$ & $0.126082784303363107778654204938\dotsc$ \\
$211$ & $0.088113526491133982480553514525\dotsc$ \\
$223$ & $0.057898365920650110981300560279\dotsc$ \\
$227$ & $0.053459670863865895829857481681\dotsc$ \\
$229$ & $0.059371895090913785790744781928\dotsc$ \\
$233$ & $0.041561490891570865670409968939\dotsc$ \\
$239$ & $0.065820141780365230014091775123\dotsc$ \\
$241$ & $0.110515313722982307702505567111\dotsc$ \\
$251$ & $0.098869364562232118021197879586\dotsc$ \\
$257$ & $0.026289614981049793532191665303\dotsc$ \\
$263$ & $0.047887172323274972467095533810\dotsc$ \\
$269$ & $0.043642785796346713737278821355\dotsc$ \\
\hline
\end{tabular}
}
\scalebox{0.5875}{
\begin{tabular}{|c|c|}
\hline
$q$  &  $m_q$\\ \hline
 $271$ & $0.072815958227510407092347024304\dotsc$ \\
$277$ & $0.079967970098330852132644457819\dotsc$ \\
$281$ & $0.060062168588754157048255824462\dotsc$ \\
$283$ & $0.003150668094148233344534460700\dotsc$ \\
$293$ & $0.135103364633442165799200521286\dotsc$ \\
$307$ & $0.080361924803122609976003922642\dotsc$ \\ 
$311$ & $0.020809138839835850375607303360\dotsc$ \\ 
$313$ & $0.038152433030706606745433863772\dotsc$ \\ 
$317$ & $0.120310482217012060618479876460\dotsc$ \\ 
$331$ & $0.090170220408966373277578719363\dotsc$ \\ 
$337$ & $0.075619767636542281468700635945\dotsc$ \\ 
$347$ & $0.061562527980707029174085407677\dotsc$ \\ 
$349$ & $0.024659800399032065009402532442\dotsc$ \\ 
$353$ & $0.097060672416567490052208231772\dotsc$ \\ 
$359$ & $0.159565297851633355627166678429\dotsc$ \\ 
$367$ & $0.120752854906502520642266950975\dotsc$ \\ 
$373$ & $0.036684373998069703748701756310\dotsc$ \\ 
$379$ & $0.062512783798157835888293219183\dotsc$ \\ 
$383$ & $0.072466929059169901487523581460\dotsc$ \\ 
$389$ & $0.035987070773221029949860780588\dotsc$ \\ 
$397$ & $0.027227952015046458842733852917\dotsc$ \\ 
$401$ & $0.104399057003481324855914310851\dotsc$ \\ 
$409$ & $0.046194322516358621884845050091\dotsc$ \\ 
$419$ & $0.035877133426605148541131058899\dotsc$ \\ 
$421$ & $0.092642081105082285944635638338\dotsc$ \\ 
$431$ & $0.060610872703868562233092295828\dotsc$ \\ 
$433$ & $0.045778529517913654069080269487\dotsc$ \\ 
$439$ & $0.108439335102642353657462341324\dotsc$ \\ 
$443$ & $0.034907168270041629941342334682\dotsc$ \\ 
$449$ & $0.068236065877290568789408277549\dotsc$ \\ 
$457$ & $0.111566405402641079895678728654\dotsc$ \\ 
$461$ & $0.059145596894312275391295323265\dotsc$ \\ 
$463$ & $0.053641121975254286882026544921\dotsc$ \\ 
$467$ & $0.082747174213806691900849070014\dotsc$ \\ 
$479$ & $0.053192268343111159252876782608\dotsc$ \\ 
$487$ & $0.072015242309793507430557740674\dotsc$ \\ 
$491$ & $0.088859444946655364010425676492\dotsc$ \\ 
$499$ & $0.036051482588918952743416502474\dotsc$ \\ 
$503$ & $0.101930324454999846629303155693\dotsc$ \\ 
$509$ & $0.046641017175720556427264466580\dotsc$ \\ 
$521$ & $0.064968723343215369312393768633\dotsc$ \\ 
$523$ & $0.049490019983931771973278347945\dotsc$ \\ 
$541$ & $0.078201802572578301759951022229\dotsc$ \\ 
$547$ & $0.011446072603108451883833352085\dotsc$ \\ 
$557$ & $0.032649429863770489497073765816\dotsc$ \\ 
$563$ & $0.040012790875967792748467993796\dotsc$ \\ 
$569$ & $0.054104318075237066903173731505\dotsc$ \\ 
$571$ & $0.098800274926131853642566583417\dotsc$ \\ 
$577$ & $0.057616230049140599007139782512\dotsc$ \\ 
$587$ & $0.028453221139884652469833267724\dotsc$ \\ 
$593$ & $0.040633070972329746592792414650\dotsc$ \\ 
$599$ & $0.040806120000547802001712541751\dotsc$ \\ 
$601$ & $0.046479560640748491238754689477\dotsc$ \\ 
$607$ & $0.053431321823469925598884667480\dotsc$ \\ 
$613$ & $0.026979398051501163961309875661\dotsc$ \\ 
$617$ & $0.086455608244385173808632307637\dotsc$ \\ 
\hline
\end{tabular}
}
\scalebox{0.5875}{
\begin{tabular}{|c|c|}
\hline
$q$  &  $m_q$\\ \hline
$619$ & $0.036809001877106866180335915559\dotsc$ \\ 
$631$ & $0.021740655876972247835117056143\dotsc$ \\ 
$641$ & $0.052039782792164018630583875044\dotsc$ \\ 
$643$ & $0.094730957910075796006995797872\dotsc$ \\ 
$647$ & $0.022707532347681205317339648994\dotsc$ \\ 
$653$ & $0.010979306879031359219927573518\dotsc$ \\ 
$659$ & $0.046308015595165386046394467522\dotsc$ \\ 
$661$ & $0.103916589731771042952256481790\dotsc$ \\ 
$673$ & $0.049755263776059657483605717979\dotsc$ \\ 
$677$ & $0.057693180910059875967225881005\dotsc$ \\ 
$683$ & $0.045197591824286077977576512346\dotsc$ \\ 
$691$ & $0.027389773133139679131666108333\dotsc$ \\ 
$701$ & $0.025792063381175160731735442054\dotsc$ \\ 
$709$ & $0.034849613800713838295949488339\dotsc$ \\ 
$719$ & $0.032014828699399819280375472745\dotsc$ \\ 
$727$ & $0.046548618667915043102907340089\dotsc$ \\ 
$733$ & $0.042139831461577750336209045957\dotsc$ \\ 
$739$ & $0.015708957869736012440269724376\dotsc$ \\ 
$743$ & $0.116540873617071834980829355749\dotsc$ \\ 
$751$ & $0.152061136399012561767846170124\dotsc$ \\ 
$757$ & $0.056975681664997732706139153050\dotsc$ \\ 
$761$ & $0.075261045249458606915771017644\dotsc$ \\ 
$769$ & $0.013153110800029449125939080369\dotsc$ \\ 
$773$ & $0.030651359792302555174852939805\dotsc$ \\ 
$787$ & $0.019038244204322518422127334614\dotsc$ \\ 
$797$ & $0.021567891774257723672469454215\dotsc$ \\ 
$809$ & $0.063775218225541117846581797608\dotsc$ \\ 
$811$ & $0.092226442843305532398212849968\dotsc$ \\ 
$821$ & $0.063566957442518142038599317693\dotsc$ \\ 
$823$ & $0.047289370792256898594184490183\dotsc$ \\ 
$827$ & $0.012749597279993981378378767012\dotsc$ \\ 
$829$ & $0.025998031567165045170854485526\dotsc$ \\ 
$839$ & $0.084436292109806136993754070909\dotsc$ \\ 
$853$ & $0.058287674110997917235105303450\dotsc$ \\ 
$857$ & $0.072425046291110298533023943432\dotsc$ \\ 
$859$ & $0.035678865131857745826151075766\dotsc$ \\ 
$863$ & $0.046770144442337100055402418780\dotsc$ \\ 
$877$ & $0.035293865694185936926398081489\dotsc$ \\ 
$881$ & $0.054634183425280742422051742979\dotsc$ \\ 
$883$ & $0.037633242851010750275554025237\dotsc$ \\ 
$887$ & $0.022554269438627487049925937822\dotsc$ \\ 
$907$ & $0.052643295498117140364919564980\dotsc$ \\ 
$911$ & $0.035463492708556820560659148802\dotsc$ \\ 
$919$ & $0.014673360300950560145530422161\dotsc$ \\ 
$929$ & $0.030444805205138918823987682845\dotsc$ \\ 
$937$ & $0.036077911931189997653707966581\dotsc$ \\ 
$941$ & $0.011109552462842771175399709306\dotsc$ \\ 
$947$ & $0.064185134239730262251762819034\dotsc$ \\ 
$953$ & $0.046995905606263748103687560739\dotsc$ \\ 
$967$ & $0.015161604253911836558220368388\dotsc$ \\ 
$971$ & $0.012974123428335617422790144852\dotsc$ \\ 
$977$ & $0.016784091339428433806510647310\dotsc$ \\ 
$983$ & $0.041396365426989204577803463342\dotsc$ \\ 
$991$ & $0.030826538808886821305709019454\dotsc$ \\ 
$997$ & $0.032897666762473802681213392412\dotsc$ \\ 
\phantom{} & \phantom{}  \\ 
\hline
\end{tabular}
}
\caption{\label{table1}
Values of   $m_q$ for every odd prime up to $1000$ with 
$30$-digit precision; computed with PARI/GP, v.~2.11.4 . Total computation time: 3 min.,  13 sec.,  583 millisec. on the Dell
Optiplex machine mentioned before.}
\end{table} 
\newpage

\begin{figure} [H]  
\begin{minipage}{0.48\textwidth}
\centerline{{\tiny $m_q \le 1.1 /\sqrt{q}$, $3\le q\le 100$}}
\includegraphics[scale=0.625,angle=0]{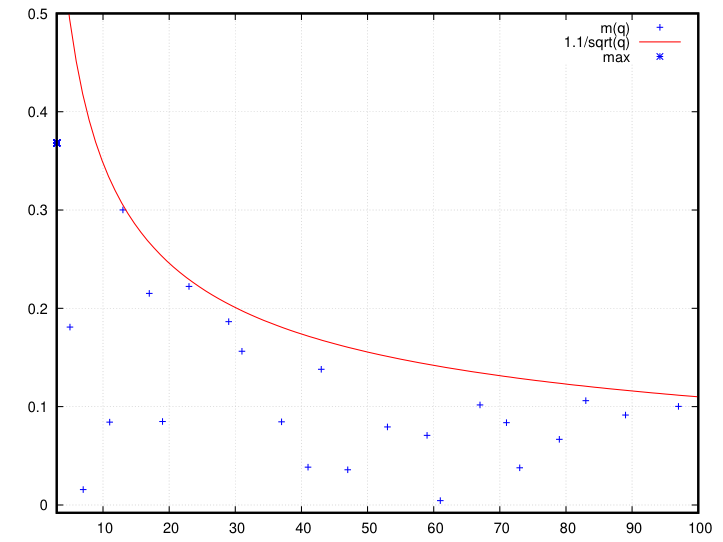}
\end{minipage}
\begin{minipage}{0.48\textwidth}
\centerline{{\tiny $m_q \le 4.25 /\sqrt{q}$, $100\le q\le 10^3$}}
\includegraphics[scale=0.625,angle=0]{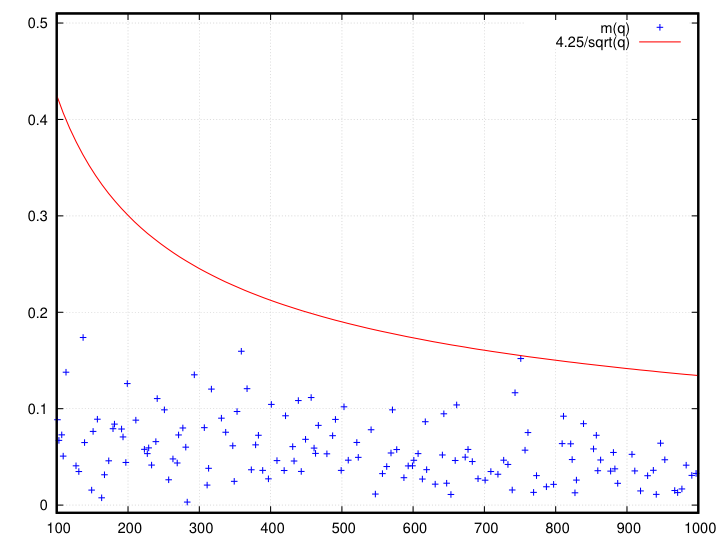}
\end{minipage} 

\vskip1.5cm
\begin{minipage}{0.48\textwidth}
\centerline{{\tiny $m_q \le 4 /\sqrt{q}$, $10^3\le q\le 10^4$}}
\includegraphics[scale=0.625,angle=0]{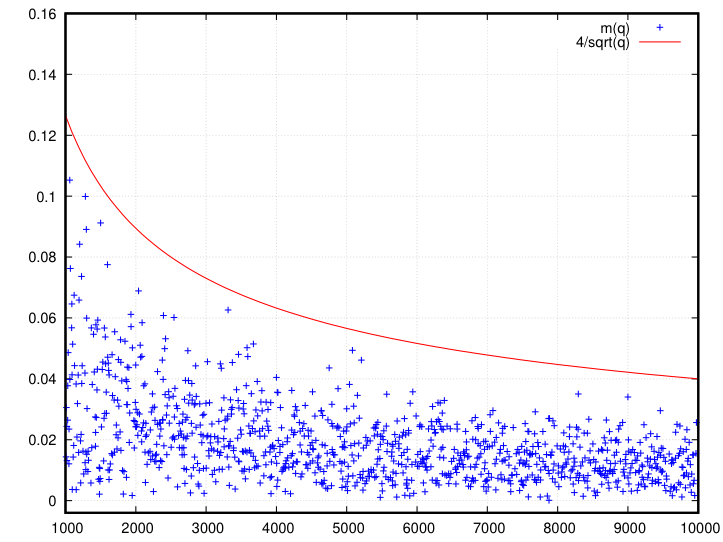}
\end{minipage}
\begin{minipage}{0.48\textwidth}
\centerline{{\tiny $m_q \le 4.2 /\sqrt{q}$, $10^4\le q\le 10^5$}}
\includegraphics[scale=0.625,angle=0]{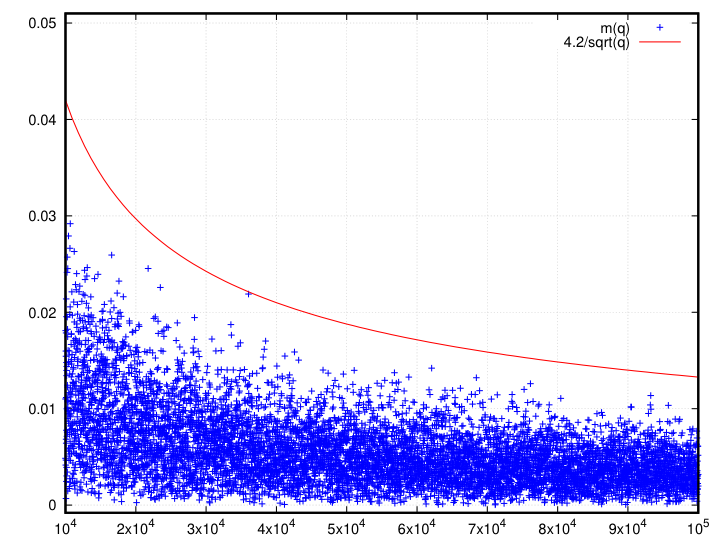}
\end{minipage} 

\vskip1.5cm
\begin{minipage}{0.48\textwidth}
\centerline{{\tiny $m_q \le 5 /\sqrt{q}$, $10^5\le q\le 10^6$}}
\includegraphics[scale=0.625,angle=0]{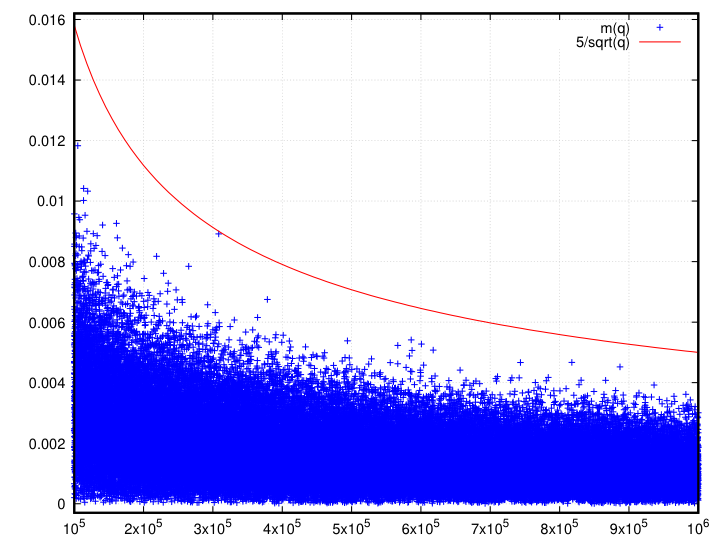}
\end{minipage}
\begin{minipage}{0.48\textwidth}
\centerline{{\tiny $m_q \le 4.85 /\sqrt{q}$, $10^6\le q\le 10^7$}}
\includegraphics[scale=0.625,angle=0]{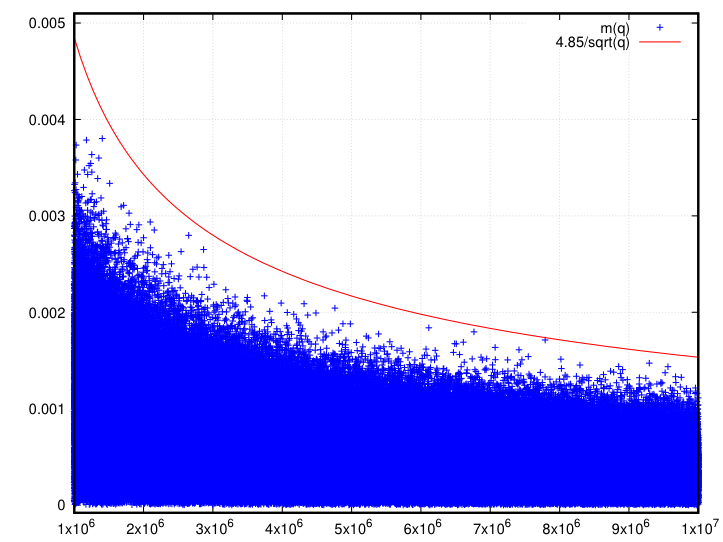}
\end{minipage} 
\caption{The values of $m_q$, $q$ prime, $3\le q\le 10^7$. 
$m_3=0.368281\dotsc$ is the maximal value. The red lines 
 represent  the function $c /\sqrt{q}$ for several values of $c$. 
 }
\label{fig1} 
\end{figure}

\begin{figure} [H]  
\begin{minipage}{0.48\textwidth}
\centerline{{\tiny $m_q^\prime :=\frac{200} {21}q m_q<500$, $3\le q\le 100$}}
\includegraphics[scale=0.625,angle=0]{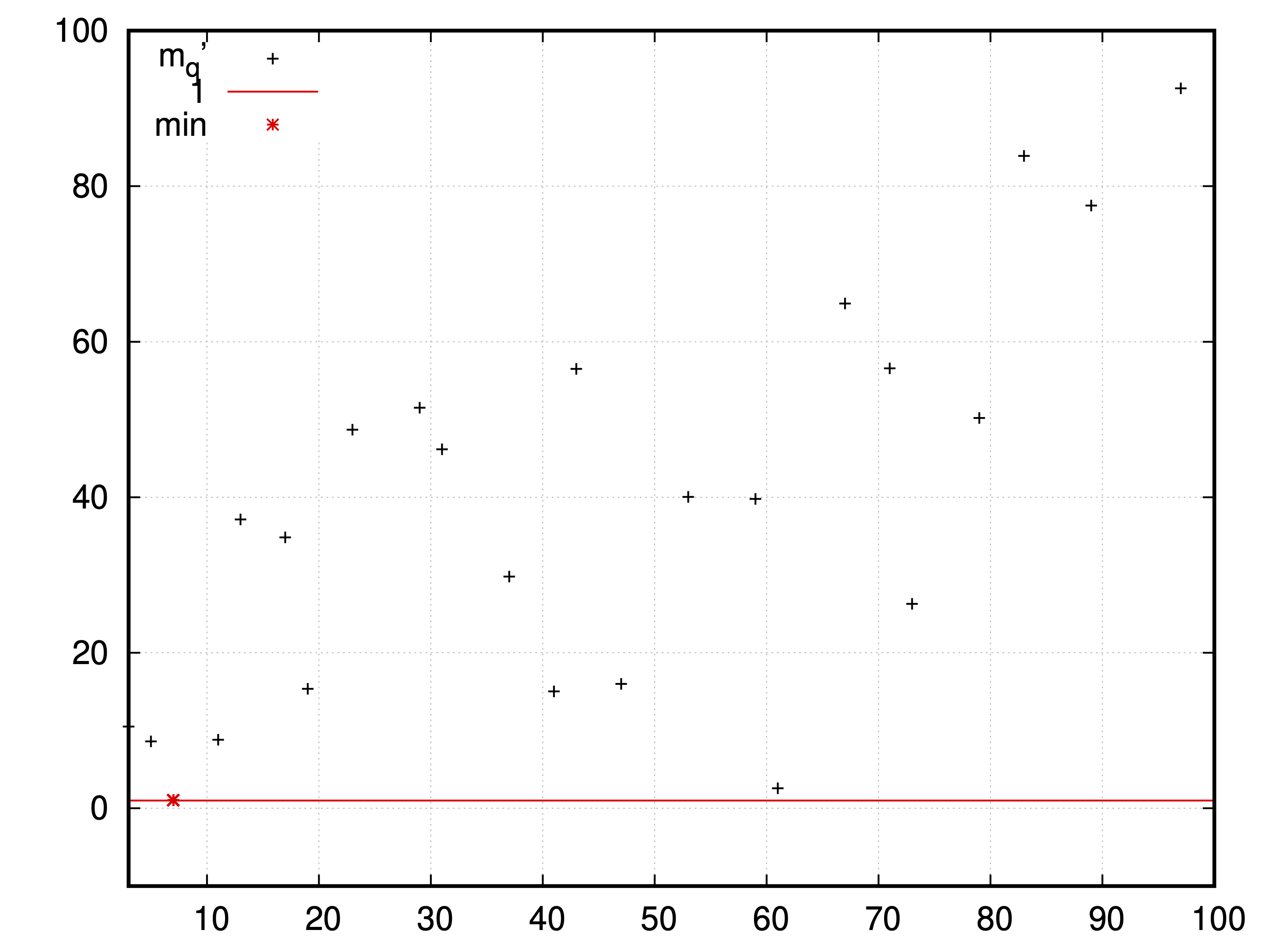}
\end{minipage}
\begin{minipage}{0.48\textwidth}
\centerline{{\tiny $m_q^\prime :=\frac{200} {21}q m_q<500$, $100\le q\le 10^3$}}
\includegraphics[scale=0.625,angle=0]{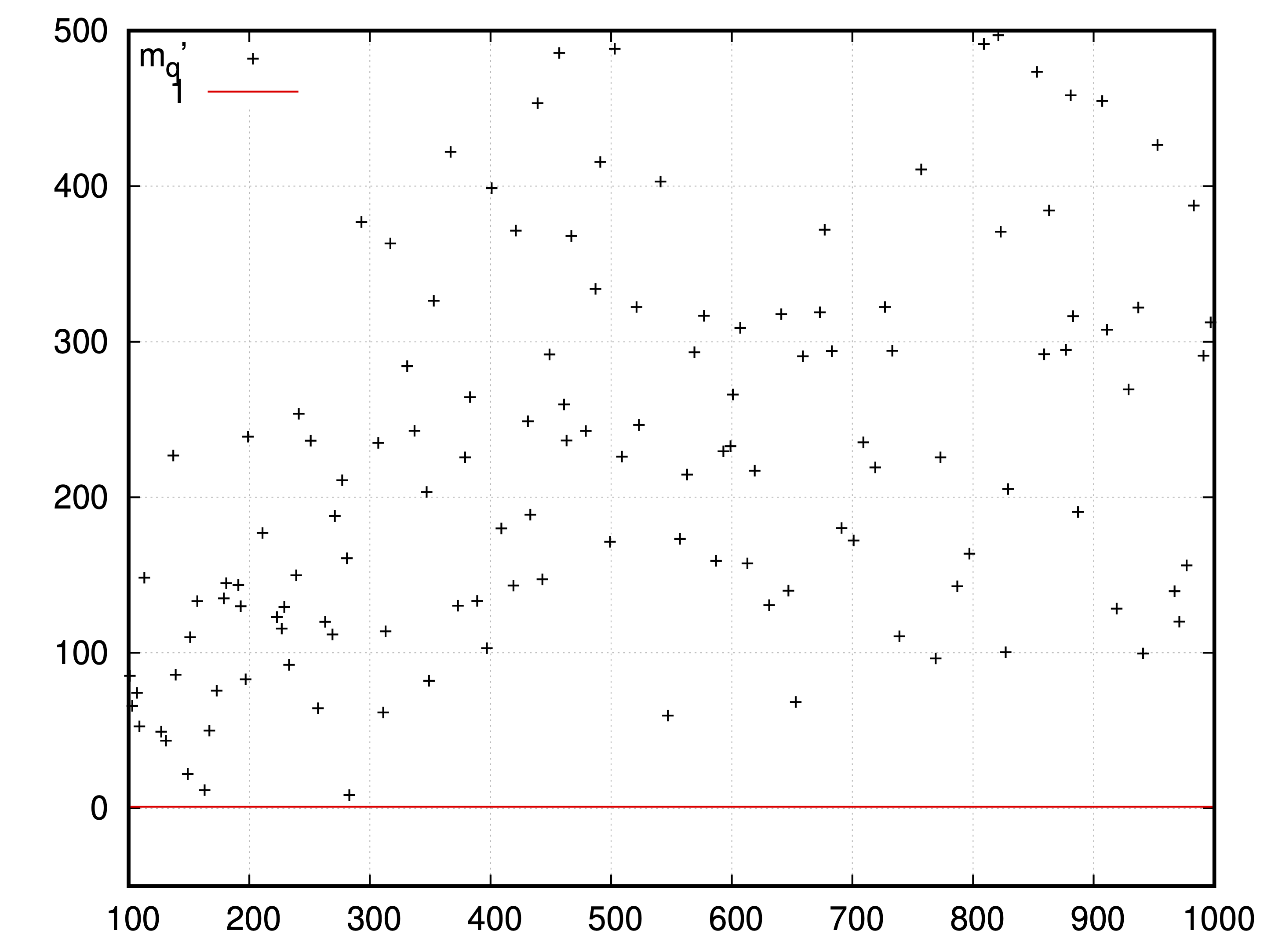}
\end{minipage} 

\vskip1.5cm
\begin{minipage}{0.48\textwidth}
\centerline{{\tiny $m_q^\prime :=\frac{200} {21}q m_q<500$, $10^3\le q\le 10^4$}}
\includegraphics[scale=0.625,angle=0]{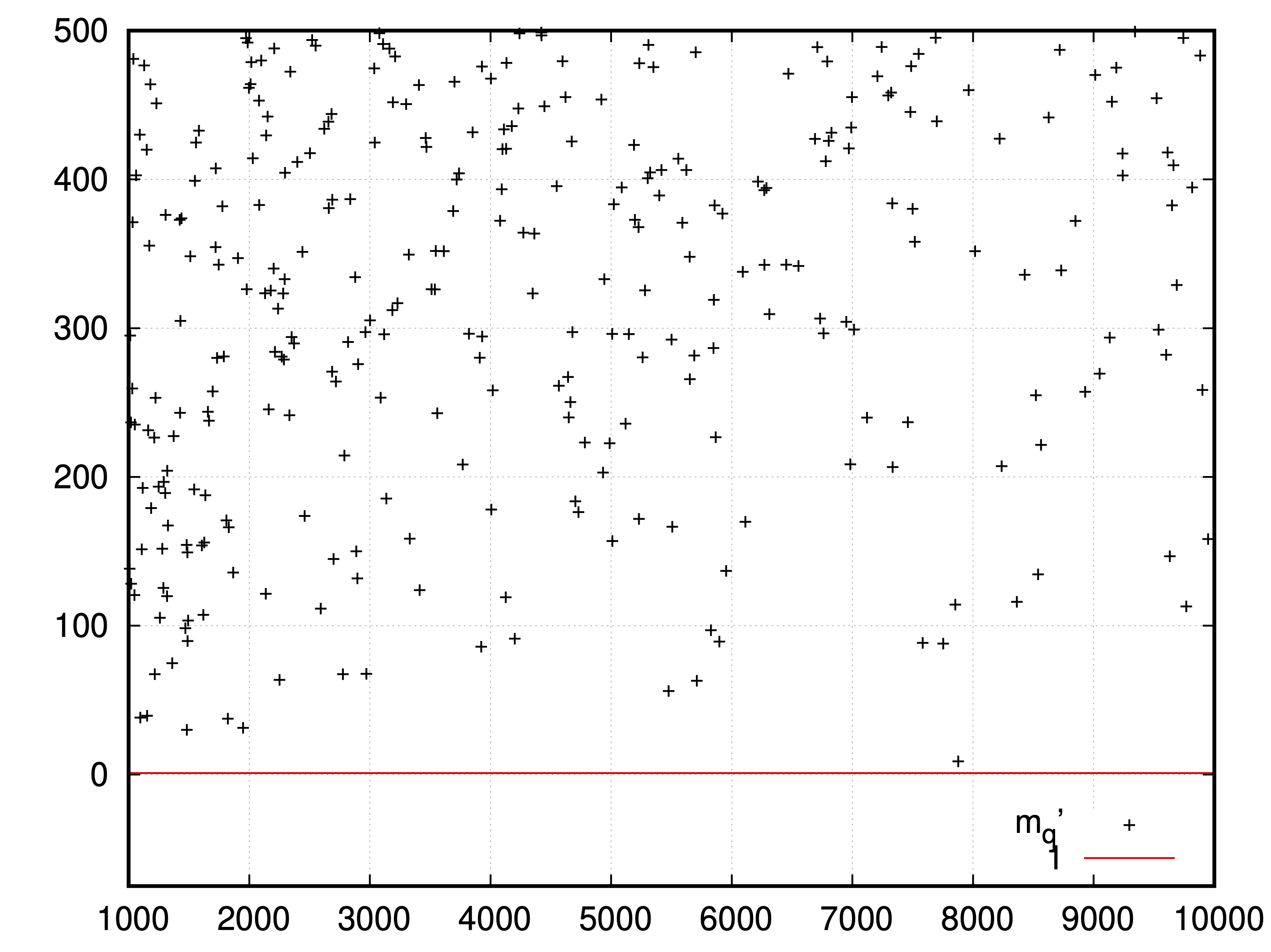}
\end{minipage}
\begin{minipage}{0.48\textwidth}
\centerline{{\tiny $m_q^\prime :=\frac{200} {21}q m_q<500$, $10^4\le q\le 10^5$}}
\includegraphics[scale=0.625,angle=0]{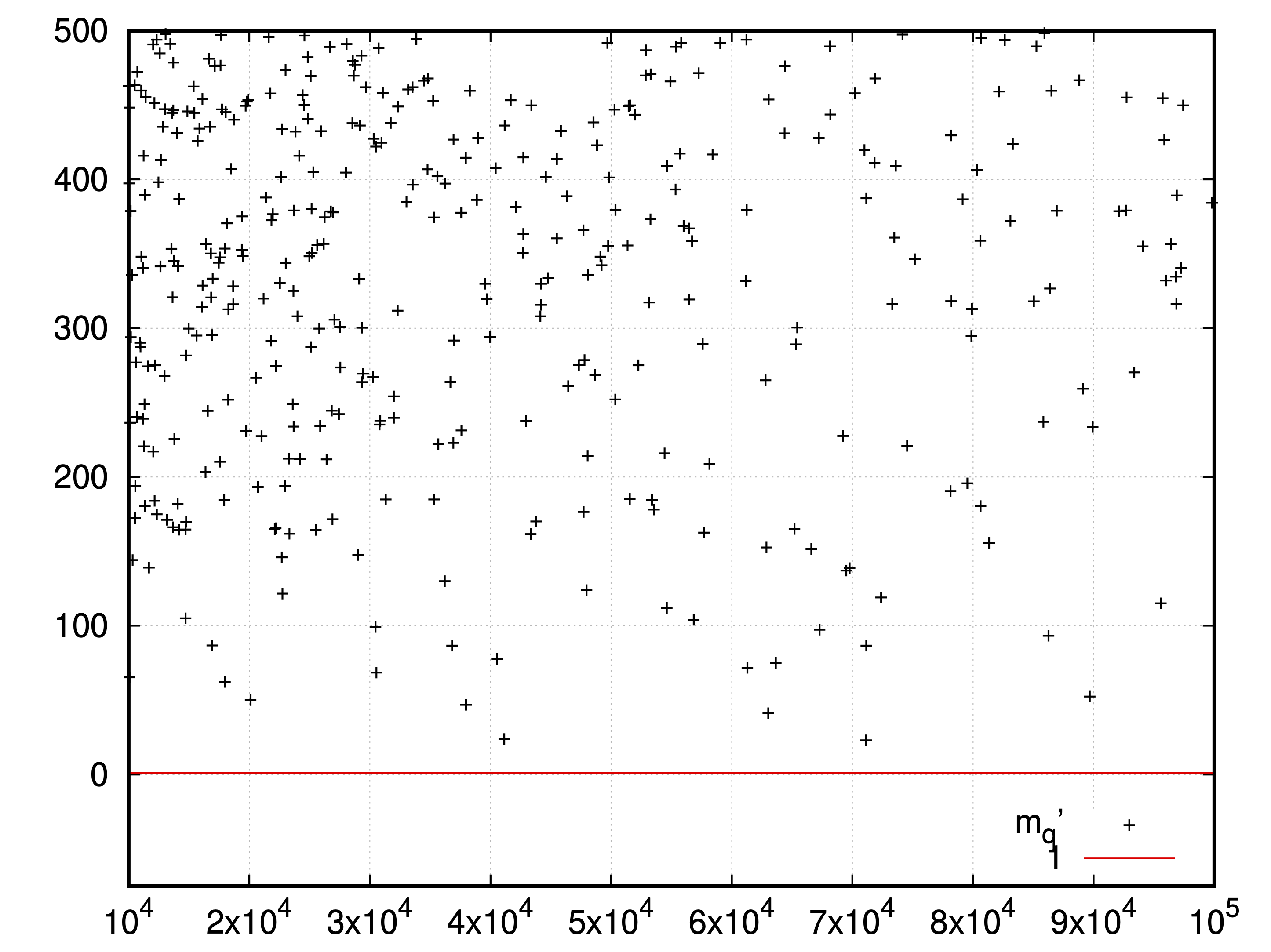}
\end{minipage} 

\vskip1.5cm
\begin{minipage}{0.48\textwidth}
\centerline{{\tiny $m_q^\prime :=\frac{200} {21}q m_q<500$, $10^5\le q\le 10^6$}}
\includegraphics[scale=0.625,angle=0]{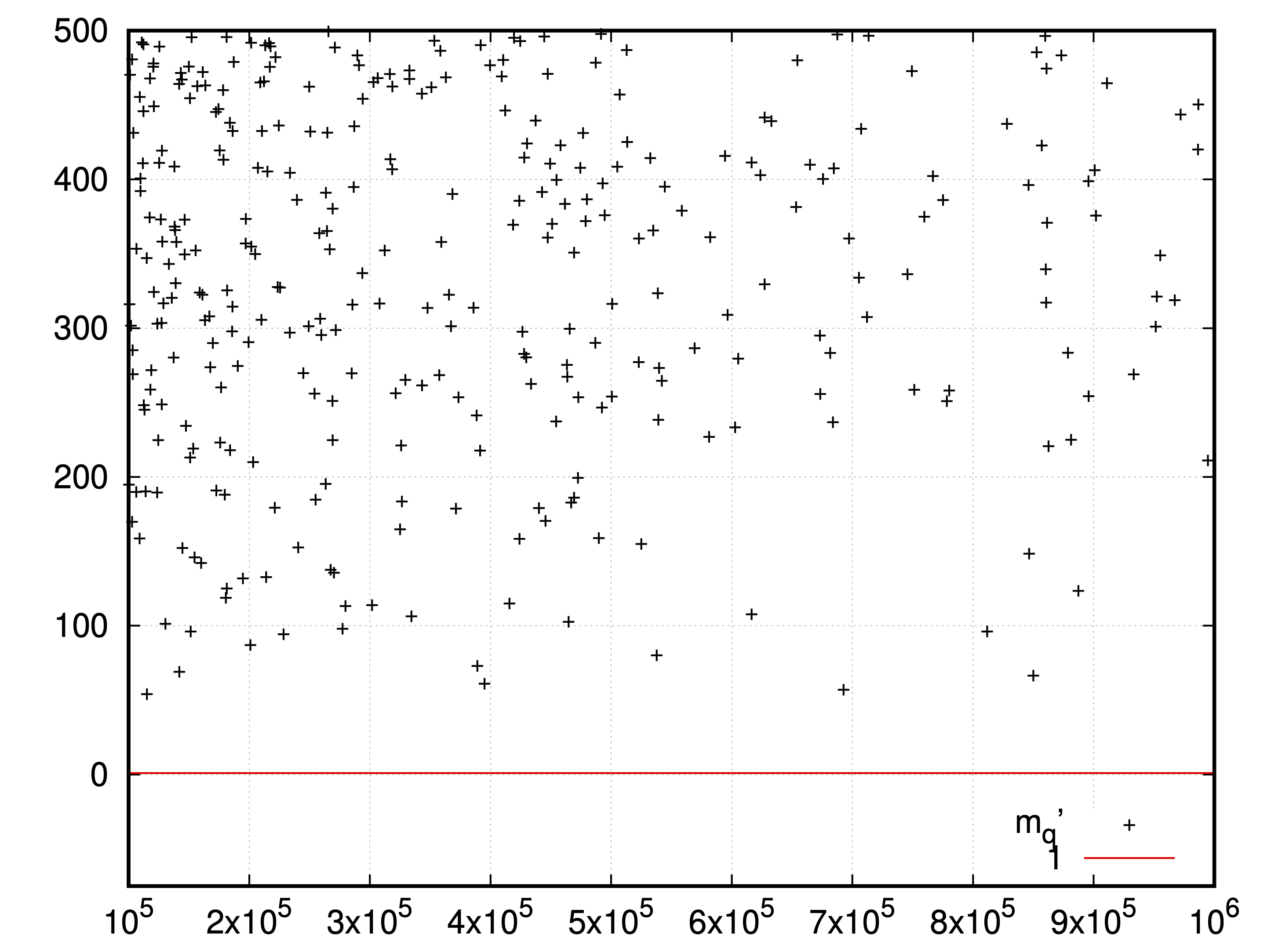}
\end{minipage}
\begin{minipage}{0.48\textwidth}
\centerline{{\tiny $m_q^\prime :=\frac{200} {21}q m_q<500$, $10^6\le q\le 10^7$}}
\includegraphics[scale=0.625,angle=0]{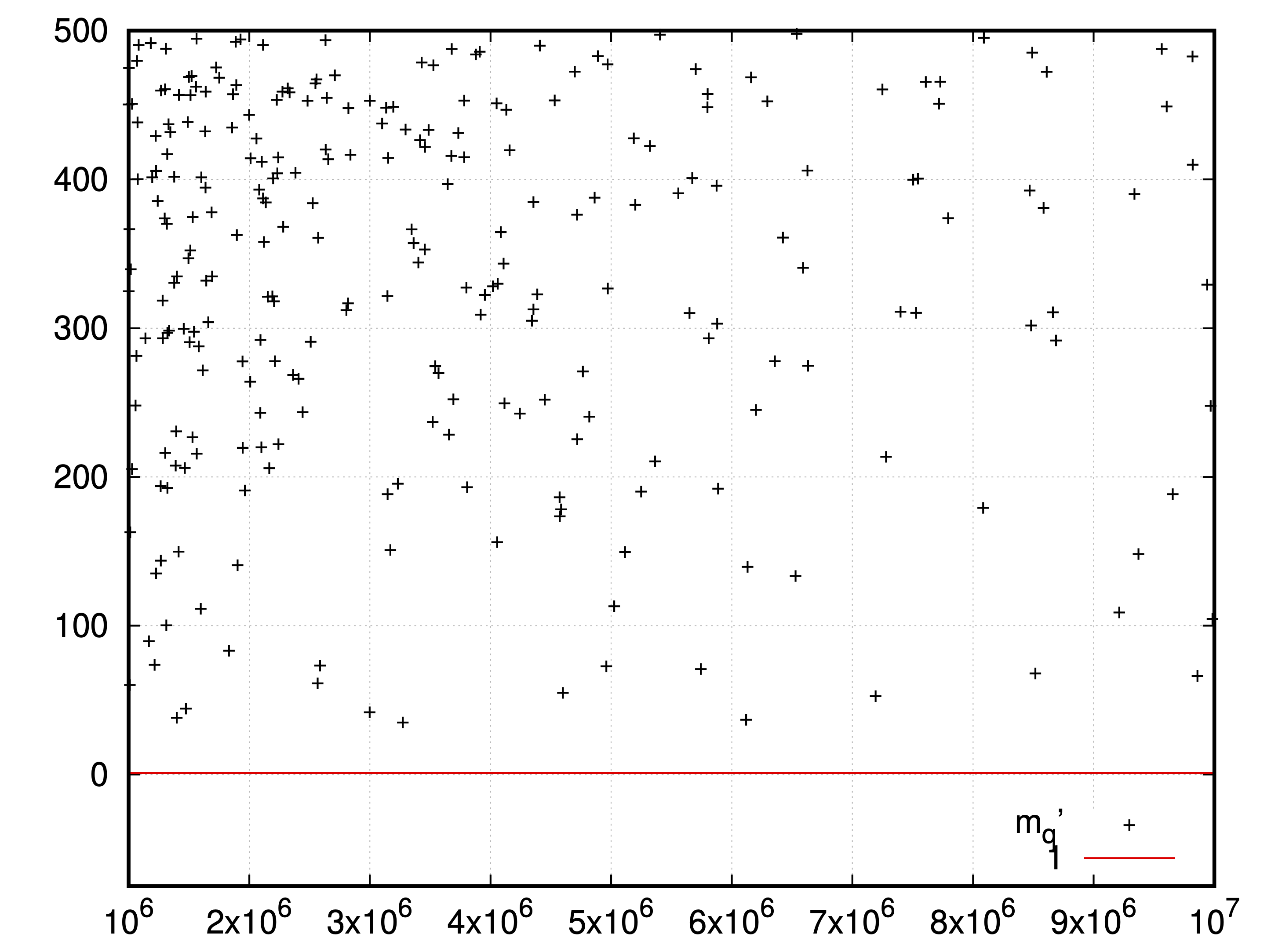}
\end{minipage} 
\caption{The values of $m_q^\prime :=\frac{200} {21}q m_q$, $q$ prime, $3\le q\le 10^7$.
The red line represents the constant function $1$.  
 }
\label{fig2} 
\end{figure}

\begin{figure}[ht] 
\includegraphics[scale=1.6,angle=90]{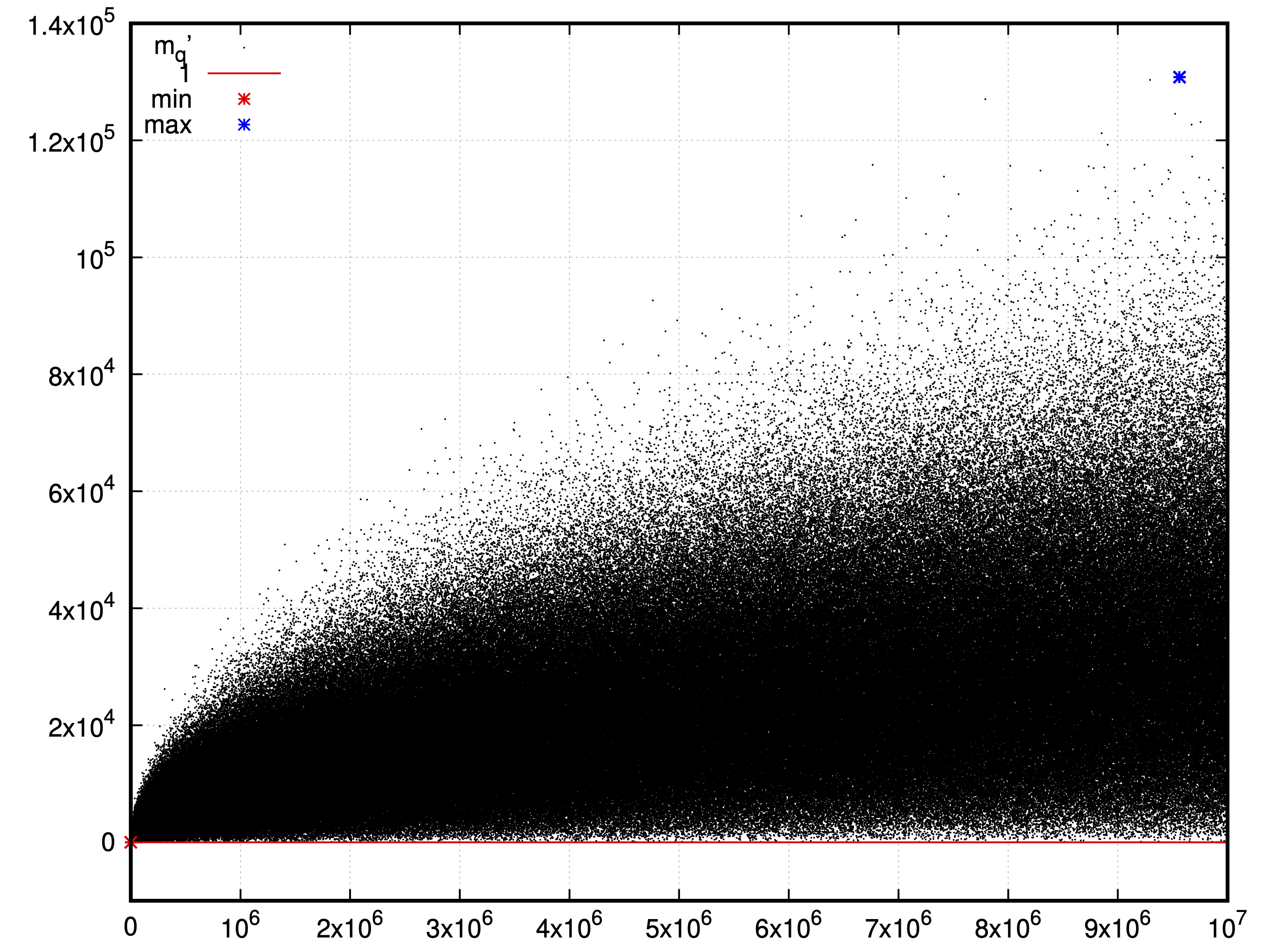}  
\caption{The values of $m_q^\prime :=\frac{200} {21}q m_q$, $q$ prime, $3\le q\le 10^7$.
The red line represents the constant function $1$. 
The minimal value for $m_q^\prime$
is $1.042379\dots$ and it is attained at $q=7$.
The maximal value for $m_q^\prime$
is $130782.760597\dots$ and it is attained at $q=9561863$. 
 }
\label{fig3} 
\end{figure}
 
\end{document}